\documentclass[12pt,reqno]{amsart}
\usepackage{amssymb,amsmath,amsthm,hyperref}
\newcommand{\sg}{\textnormal{sg}}

\newtheorem{theorem}{Theorem}
\newtheorem{lemma}[theorem]{Lemma}
\newtheorem{corollary}[theorem]{Corollary}
\newtheorem{proposition}[theorem]{Proposition}

\theoremstyle{remark}

\theoremstyle{definition}

\numberwithin{theorem}{section} 
\numberwithin{equation}{section}
\numberwithin{example}{section}

\title{A heuristic guide to evaluating triple-sums}

\author{Eric T. Mortenson}

\begin{document}

\date{21 December 2020}



\begin{abstract}
Using a heuristic that relates Appell--Lerch functions to divergent partial theta functions one can expand Hecke-type double-sums in terms of Appell--Lerch functions.  We give examples where the heuristic can be used as a guide to evaluate analogous triple-sums in terms of Appell--Lerch functions or false theta functions.
\end{abstract}

\address{Saint Petersburg State University}
\email{etmortenson@gmail.com}
\maketitle
\setcounter{section}{-1}


\section{Notation}

 Let $q$ be a nonzero complex number with $|q|<1$ and define $\mathbb{C}^*:=\mathbb{C}-\{0\}$.  Recall
\begin{gather}
(x)_n=(x;q)_n:=\prod_{i=0}^{n-1}(1-q^ix), \ \ (x)_{\infty}=(x;q)_{\infty}:=\prod_{i=0}^{\infty}(1-q^ix),\notag \\
j(x;q):=(x)_{\infty}(q/x)_{\infty}(q)_{\infty}=\sum_{n=-\infty}^{\infty}(-1)^nq^{\binom{n}{2}}x^n,\notag
\end{gather}
and
\begin{equation*}
j(x_1,x_2,\dots,x_n;q):=j(x_1;q)j(x_2,q)\cdots j(x_n;q),
\end{equation*}
where in the penultimate line the equivalence of product and sum follows from Jacobi's triple product identity.    Let $a$ and $m$ be integers with $m$ positive.  Define
\begin{gather*}
J_{a,m}:=j(q^a;q^m), \ \ J_m:=J_{m,3m}=\prod_{i=1}^{\infty}(1-q^{mi}), \ {\text{and }}\overline{J}_{a,m}:=j(-q^a;q^m).
\end{gather*}
We will use the following definition of an Appell--Lerch function \cite{HM}
\begin{equation}
m(x,q,z):=\frac{1}{j(z;q)}\sum_{r=-\infty}^{\infty}\frac{(-1)^rq^{\binom{r}{2}}z^r}{1-q^{r-1}xz}.\label{equation:mdef-eq}
\end{equation}

\section{Introduction}
Appell--Lerch functions are the building blocks of Ramanujan's classical mock theta functions \cite[Section 5]{HM}.  For two of Ramanujan's fifth order mock theta functions \cite[Section 5]{HM}, \cite{AG, H1}:
{\allowdisplaybreaks \begin{align}
\chi_0(q):&=\sum_{n\ge 0}\frac{q^n}{(q^{n+1})_n}= 2-3m(q^7,q^{15},q^9)-3q^{-1}m(q^2,q^{15},q^4)+\frac{2J_5^2J_{2,5}}{J_{1,5}^2},\label{equation:5th-chi0(q)}\\
\chi_1(q):&=\sum_{n\ge 0}\frac{q^n}{(q^{n+1})_{n+1}}= -3q^{-1}m(q^4,q^{15},q^3)-3q^{-2}m(q,q^{15},q^2)-\frac{2J_5^2J_{1,5}}{J_{2,5}^2}. \label{equation:5th-chi1(q)}
\end{align}}%

In recent work \cite{HM}, we used a heuristic (\ref{equation:mxqz-heuristic}) to express Hecke-type double-sums of the form \cite[$(1.4)$]{HM}
\begin{equation}
\Big (\sum_{r,s\ge 0}-\sum_{r,s<0}\Big )(-1)^{r+s}x^ry^sq^{a\binom{r}{2}+brs+c\binom{s}{2}}
\label{equation:fabc-def-intro},
\end{equation}
where $a$, $b$, $c$ are strictly positive integers, in terms of theta functions and the $m(x,q,z)$ function.  As an example, a special case of more general results from \cite[$(1.7)$]{HM} reads
\begin{align}
&\Big (\sum_{r,s\ge 0}-\sum_{r,s<0}\Big )(-1)^{r+s}x^ry^sq^{\binom{r}{2}+2rs+\binom{s}{2}}\label{equation:f121}\\
&\ \ \ \ \ =j(y;q)m\Big (\frac{q^2x}{y^2},q^3,-1\Big )+j(x;q)m\Big (\frac{q^2y}{x^2},q^3,-1\Big )
- \frac{yJ_{3}^3j(-x/y;q)j(q^2xy;q^3)}{\overline{J}_{0,3}j(-qy^2/x,-qx^2/y;q^3)}.\notag 
\end{align}

Objects that could be thought of as Hecke-type triple-sums have been the subject of recent work.  Zwegers discovered that \cite[Theorem $1$]{Zw}
{\allowdisplaybreaks \begin{align}
2-\frac{1}{(q)_{\infty}^2}\Big ( \sum_{k,l,m\ge 0}+\sum_{k,l,m<0}\Big )(-1)^{k+l+m}
q^{\binom{k}{2}+\binom{l}{2}+\binom{m}{2}+2kl+2km+2lm+k+l+m}&=\chi_0(q),\label{equation:chi0-triple}\\
\frac{1}{(q)_{\infty}^2}\Big ( \sum_{k,l,m\ge 0}+\sum_{k,l,m<0}\Big )(-1)^{k+l+m}
q^{\binom{k}{2}+\binom{l}{2}+\binom{m}{2}+2kl+2km+2lm+2k+2l+2m}&=\chi_1(q),\label{equation:chi1-triple}
\end{align}}%
which suggests considering the following building block for triple-sums:
\begin{equation}
\mathfrak{g}_{a,b,c,d,e,f}(x,y,z,q):=\Big ( \sum_{r,s,t\ge 0}+\sum_{r,s,t<0}\Big )(-1)^{r+s+t}x^{r}y^{s}z^{t}
q^{a\binom{r}{2}+brs+c\binom{s}{2}+drt+est+f\binom{t}{2}},\label{equation:triple-sum-def}
\end{equation}
where $a$, $b$, $c$, $d$, $e$, and $f$ are strictly positive integers.  For related examples, see \cite{BRZ, M2017}.  Indefinite binary theta series for $\chi_0(q)$ and $\chi_1(q)$ can be found in \cite{G,Za}.

Kim and Lovejoy \cite{KL} recently found examples of triple-sums that evaluate to false theta functions.  False theta functions are theta functions but with the wrong signs \cite{AW}.  In particular, they found 

{\allowdisplaybreaks
\begin{align}
 \mathfrak{g}_{1,7,1,1,1,1}&(q^2,q^3,q,q)+q^4\mathfrak{g}_{1,7,1,1,1,1}(q^6,q^7,q^2,q)
=J_1^2\sum_{r=0}^{\infty}(-1)^rq^{3r(r+1)/2},\label{equation:KL-id1}\\
\mathfrak{g}_{1,5,1,1,1,1}&(q^2,q^2,q,q)+q^3\mathfrak{g}_{1,5,1,1,1,1}(q^5,q^5,q^2,q) 
 =J_1J_{1,2}\sum_{r = 0}^{\infty}(-1)^rq^{3r^2+2r}(1+q^{2r+1}).\label{equation:KL-id2}
 \end{align}}%


In Section \ref{section:prelim}, we introduce a few basic facts and preliminary results.  In Section \ref{section:preview}, we remind the reader how the heuristic is used to understand the (mock) modularity of double-sums.   In Section \ref{section:obtain-the-triple}, we obtain the relation (\ref{equation:HeckeTriple-id}) that will be used to understand the modularity of triple-sums.   We will also introduce a few basic facts about triple-sums.  

In Section \ref{section:Zwegers}, we consider the mock theta functions $\chi_0(q)$ and $\chi_1(q)$.  We use the relation (\ref{equation:HeckeTriple-id}) to demonstrate how the heuristic (\ref{equation:mxqz-heuristic}) guides us from their respective triple-sums in $(\ref{equation:chi0-triple})$ and $(\ref{equation:chi1-triple})$ to their Appell--Lerch sum expressions modulo a theta function, see $(\ref{equation:5th-chi0(q)})$ and $(\ref{equation:5th-chi1(q)})$.  In Section \ref{section:KimLovejoy}, we demonstrate how the relation (\ref{equation:HeckeTriple-id}) guides us from Kim and Lovejoy's two triple-sums in $(\ref{equation:KL-id1})$ and $(\ref{equation:KL-id2})$ to their respective false theta function forms.   The work in Sections \ref{section:Zwegers} and  \ref{section:KimLovejoy} is experimental and does not contain proofs,  but perhaps more general formulas such as those found in \cite{BRZ, HM, M2017} are lurking in the background.

In Section \ref{section:exploration}, we use triple-sum relations from Section \ref{section:obtain-the-triple} in order to prove corollaries to Identities (\ref{equation:chi0-triple}) and (\ref{equation:chi1-triple}).  We show
\begin{corollary}\label{corollary:newIds} We have
\begin{gather}
\mathfrak{g}_{1,2,1,2,2,1}(q^3,q^3,q^3,q)=0,\label{equation:newId-1}\\
\mathfrak{g}_{1,2,1,2,2,1}(q,q,q^2,q)=J_{1}^2,\label{equation:newId-2}\\
\mathfrak{g}_{1,2,1,2,2,1}(q,q^2,q^2,q)=J_{1}^2\chi_0(q),\label{equation:newId-3}\\
\mathfrak{g}_{1,2,1,2,2,1}(q,q,q^3,q)=J_1^2(q\chi_0(q)+1-q),\label{equation:newId2-1}\\
\mathfrak{g}_{1,2,1,2,2,1}(q,q^2,q^3,q)=J_1^{2}(q\chi_1(q)+1).\label{equation:newId2-2}
\end{gather}
\end{corollary}
One wonders how triple-sums whose parameters differ slightly from the left-hand sides of (\ref{equation:chi0-triple}) and (\ref{equation:chi1-triple}) might behave.  In Section \ref{section:conjecture}, we demonstrate how our heuristic methods lead us to suspect
\begin{align}
\mathfrak{g}_{1,3,1,3,3,1}&(q,q,q,q)\label{equation:g131331-conjecture}\\
&= 3J_{1,2}\overline{J}_{3,8} m ( -q^{27},q^{56},-1  )
 +3q^{-2}J_{1,2}\overline{J}_{3,8} m ( -q^{13},q^{56},-1  ) \notag\\
 &\ \ \ \ \  -3q^{-7}J_{1,2}\overline{J}_{1,8} m ( -q^{-1},q^{56},-1  )
 -3q^{-16}J_{1,2}\overline{J}_{1,8} m ( -q^{-15},q^{56},-1  ) + \theta(q),
 \notag
 \end{align}
where $\theta(q)$ is a yet to be determined weight $3/2$ theta function.

\section{Preliminaries}\label{section:prelim}
We will frequently use the following product rearrangements without mention
{\allowdisplaybreaks \begin{gather}
\overline{J}_{0,1}=2\overline{J}_{1,4}=\frac{2J_2^2}{J_1},  \overline{J}_{1,2}=\frac{J_2^5}{J_1^2J_4^2},   J_{1,2}=\frac{J_1^2}{J_2},   \overline{J}_{1,3}=\frac{J_2J_3^2}{J_1J_6},\notag\\
  J_{1,4}=\frac{J_1J_4}{J_2},   J_{1,6}=\frac{J_1J_6^2}{J_2J_3},   \overline{J}_{1,6}=\frac{J_2^2J_3J_{12}}{J_1J_4J_6}.\notag
\end{gather}}%
We have the general identities:
\begin{subequations}
{\allowdisplaybreaks \begin{gather}
j(q^n x;q)=(-1)^nq^{-\binom{n}{2}}x^{-n}j(x;q), \ \ n\in\mathbb{Z},\label{equation:j-elliptic}\\
j(x;q)=j(q/x;q)=-xj(x^{-1};q)\label{equation:j-inv},\\
j(x;q)={J_1}j(x,qx,\dots,q^{n-1}x;q^n)/{J_n^n} \ \ {\text{if $n\ge 1$,}}\label{j-mod-inc}\\
j(x^n;q^n)={J_n}j(x,\zeta_nx,\dots,\zeta_n^{n-1}x;q^n)/{J_1^n} \ \ {\text{if $n\ge 1$,}}\label{j-mod-dec}
\end{gather}}%
\end{subequations}
\noindent  where $\zeta_n$ is an $n$-th primitive root of unity.   We also have from \cite[Theorem 1.1]{H1}
\begin{equation}
j(-x,q)j(y,q)+j(x,q)j(-y,q)=2j(xy,q^2)j(qx^{-1}y,q^2).\label{equation:H1Thm1.2B}
\end{equation}

The Appell--Lerch function $m(x,q,z)$ satisfies several functional equations and identities, which we collect in the form of a proposition.  

\begin{proposition}  For generic $x,z\in \mathbb{C}^*$
{\allowdisplaybreaks \begin{subequations}
\begin{gather}
m(x,q,z)=m(x,q,qz),\label{equation:mxqz-fnq-z}\\
m(x,q,z)=x^{-1}m(x^{-1},q,z^{-1}),\label{equation:mxqz-flip}\\
m(qx,q,z)=1-xm(x,q,z),\label{equation:mxqz-fnq-x}\\
m(x,q,z_1)-m(x,q,z_0)=\frac{z_0J_1^3j(z_1/z_0;q)j(xz_0z_1;q)}{j(z_0;q)j(z_1;q)j(xz_0;q)j(xz_1;q)}.\label{equation:changing-z-theorem}
\end{gather}
\end{subequations}}
\end{proposition}

The $m(x,q,z)$ function has a useful (slightly rewritten) functional equation
\begin{equation}
m(x,q,z)=1-q^{-1}xm(q^{-1}x,q,z)\label{equation:mxqz-altdef0intro}.
\end{equation}
In recent work \cite[Section $3$]{HM} we introduced a heuristic point of view which guides our study of the Appell--Lerch function $m(x,q,z)$ and Hecke-type double-sums.  If we iterate (\ref{equation:mxqz-altdef0intro}), we obtain
\begin{equation}
m(x,q,z)\sim \sum_{r\ge 0}(-1)^rq^{-\binom{r+1}{2}}x^r,\label{equation:mxqz-heuristic}
\end{equation}
where `$\sim$'  means mod theta.  We cannot use an equal sign here, because the series on the right diverges for $|q|<1$.  However, it is often useful to think of $m(x,q,z)$ as a partial theta series with $q$ replaced by $q^{-1}$.    Property (\ref{equation:changing-z-theorem}) says that changing $z$ only adjusts the theta function, hence the expression $m(x,q,*)$ when working mod theta.

We define
\begin{equation}
\mathfrak{f}_{a,b,c}(x,y,q):=\Big (\sum_{r,s\ge 0}-\sum_{r,s<0}\Big )(-1)^{r+s}x^ry^sq^{a\binom{r}{2}+brs+c\binom{s}{2}},
\label{equation:fabc-def-prelim}\\
\end{equation}
and note that we can also write
{\allowdisplaybreaks \begin{align}
\mathfrak{f}_{a,b,c}(x,y,q)& =\sum_{\sg(r)=\sg(s)}\sg(r)(-1)^{r+s}x^ry^sq^{a\binom{r}{2}+brs+c\binom{s}{2}},\\
&=\sum_{r,s}\sg(r,s)(-1)^{r+s}x^ry^sq^{a\binom{r}{2}+brs+c\binom{s}{2}},
\end{align}}%
where 
\begin{equation}
\sg(r):=
\begin{cases}
1 & r\ge0,\\
-1 & r<0,
\end{cases}
\end{equation}
and
\begin{equation}
\sg(r,s):=(\sg(r)+\sg(s))/2.
\end{equation}

Our Hecke-type double-sum notation has several useful properties.  We will later determine analogous properties for triple-sums.  For $b<a$, we follow the usual convention:
\begin{equation}
\sum_{r=a}^{b}c_{r}:=-\sum_{r=b+1}^{a-1}c_r.
\end{equation}

\begin{proposition}[Proposition 6.2 \cite{HM}] \label{proposition:f-flip} For $x,y\in\mathbb{C}^{*}$
\begin{align}
\mathfrak{f}_{a,b,c}(x,y,q)&=-\frac{q^{a+b+c}}{xy}\mathfrak{f}_{a,b,c}(q^{2a+b}/x,q^{2c+b}/y,q).\label{equation:fabc-flip}
\end{align}
\end{proposition}

\begin{proposition}[Proposition 6.3 \cite{HM}]  \label{proposition:f-functionaleqn}For $x,y\in\mathbb{C}^*$ and $\ell, k \in \mathbb{Z}$
\begin{align}
\mathfrak{f}_{a,b,c}(x,y,q)&=(-x)^{\ell}(-y)^kq^{a\binom{\ell}{2}+b\ell k+c\binom{k}{2}}\mathfrak{f}_{a,b,c}(q^{a\ell+bk}x,q^{b\ell+ck}y,q) \label{equation:f-shift}\\
&\ \ \ \ +\sum_{m=0}^{\ell-1}(-x)^mq^{a\binom{m}{2}}j(q^{mb}y;q^c)+\sum_{m=0}^{k-1}(-y)^mq^{c\binom{m}{2}}j(q^{mb}x;q^a).\notag
\end{align}
\end{proposition}

We have many expressions that evaluate Hecke-type double-sums in terms of Appell--Lerch functions \cite{HM}.  We recall one that we will use here.
\begin{theorem}\cite[Theorem $1.4$]{HM} \label{theorem:main-acdivb}Let $a,b,$ and $c$ be positive integers with $ac<b^2$ and $b$ divisible by $a$, $c$. Then
\begin{equation*}
\mathfrak{f}_{a,b,c}(x,y,q)=\mathfrak{h}_{a,b,c}(x,y,q,-1,-1)-\frac{1}{\overline{J}_{0,b^2/a-c}\overline{J}_{0,b^2/c-a}}\cdot \theta_{a,b,c}(x,y,q),
\end{equation*}
where 
\begin{align*}
\mathfrak{h}_{a,b,c}(x,y,q,z_1,z_0):&=j(x;q^a)m\Big( -q^{a\binom{b/a+1}{2}-c}{(-y)}{(-x)^{-b/a}},q^{b^2/a-c},z_1 \Big )\\
& \ \ \ \ \ \ +j(y;q^c)m\Big( -q^{c\binom{b/c+1}{2}-a}{(-x)}{(-y)^{-b/c}},q^{b^2/c-a},z_0 \Big ),
\end{align*}
and
\begin{align*}
&\theta_{a,b,c}(x,y,q):=\sum_{d=0}^{b/c-1}\sum_{e=0}^{b/a-1}\sum_{f=0}^{b/a-1}
q^{(b^2/a-c)\binom{d+1}{2}+(b^2/c-a)\binom{e+f+1}{2}+a\binom{f}{2}}j\big (q^{(b^2/a-c)(d+1)+bf}y;q^{b^2/a}\big ) \\
&\ \ \ \ \ \ \ \ \ \  \cdot(-x)^{f} j\big (q^{b(b^2/(ac)-1)(e+f+1)-(b^2/a-c)(d+1)+b^3(b-a)/(2a^2c)}(-x)^{b/a}y^{-1};q^{(b^2/a)(b^2/(ac)-1)}\big ) \\
& \cdot \frac{J_{b(b^2/(ac)-1)}^3j\big (q^{ (b^2/c-a)(e+1)+(b^2/a-c)(d+1)-c\binom{b/c}{2}-a\binom{b/a}{2}}(-x)^{1-b/a}(-y)^{1-b/c};q^{b(b^2/(ac)-1)}\big )}
{j\big (q^{(b^2/c-a)(e+1)-c\binom{b/c}{2}}(-x)(-y)^{-b/c},q^{(b^2/a-c)(d+1)-a\binom{b/a}{2}}(-x)^{-b/a}(-y);q^{b(b^2/(ac)-1)}\big )}.
\end{align*}
\end{theorem}

Theorem \ref{theorem:main-acdivb} has the following two specializations.
\begin{corollary}  \label{corollary:f441-expansion} We have
\begin{align}
\mathfrak{f}_{4,4,1}(x,y,q)&=\mathfrak{h}_{4,4,1}(x,y,q,-1,-1)\notag\\
&\ \ \ \ \ -\sum_{d=0}^3\frac{q^{3\binom{d+1}{2}}j(q^{3+3d}y;q^4)j(-q^{9-3d}x/y;q^{12})J_{12}^3j(-q^{9+3d}/y^3;q^{12})}{\overline{J}_{0,3}\overline{J}_{0,12}j(-q^{6}x/y^4;q^{12})j(q^{3+3d}y/x;q^{12})},\label{equation:f441}
\end{align}
where
\begin{equation}
\mathfrak{h}_{4,4,1}(x,y,q,-1,-1)=j(x;q^4)m\big ( -q^3y/x,q^3,-1 \big )+j(y;q) m\big ( q^6 x/y^{4}, q^{12}, -1 \big ).\label{equation:h441}
\end{equation}
\end{corollary}

\begin{corollary}  \label{corollary:f331-expansion}  We have
{\allowdisplaybreaks \begin{align}
\mathfrak{f}_{3,3,1}(x,y,q)&=\mathfrak{h}_{3,3,1}(x,y,q,-1,-1)\label{equation:f331-id}\\
&\ \ \ \ -\sum_{d=0}^{2}
\frac{q^{d(d+1)}j\big (q^{2+2d}y;q^{3}\big )  j\big (-q^{4-2d}x/y;q^{6}\big ) J_{6}^3j\big (q^{5+2d}/y^{2};q^{6}\big )}
{4\overline{J}_{2,8}\overline{J}_{6,24}j\big (q^{3}x/y^{3};q^6\big )j\big (q^{2+2d}y/x;q^6\big )},\notag
\end{align}}%
where
\begin{align}
\mathfrak{h}_{3,3,1}(x,y,q,-1,-1)=j(x;q^3)m(-q^2x^{-1}y,q^2,-1)+j(y;q)m(-q^3xy^{-3},q^6,-1).\label{equation:g331-id}
\end{align}
\end{corollary}

\begin{proposition}\label{proposition:f111-eval} \label{proposition:f111-zero} For $m,n\in \mathbb{Z}$, $m\ne n$, we have
\begin{equation}
\mathfrak{f}_{1,1,1}(q^{m},q^{n},q)=0.
\end{equation}
\end{proposition}
\begin{proof}[Proof of Proposition \ref{proposition:f111-eval}]
Using the functional equation (\ref{equation:f-shift}), we note that the two theta functions have $j(q^m;q)=j(q^n;q)=0$, giving us
\begin{equation*}
\mathfrak{f}_{1,1,1}(q^{m},q^{n},q)=(-1)^{\ell+k}q^{m\ell}q^{nk}q^{\binom{\ell}{2}+\ell k +\binom{k}{2}}
\mathfrak{f}_{1,1,1}(q^{\ell+k+m},q^{\ell+k+n},q).
\end{equation*}
If we choose $\ell=-k$ then
 \begin{equation*}
\mathfrak{f}_{1,1,1}(q^{m},q^{n},q)=q^{m\ell-n\ell}\mathfrak{f}_{1,1,1}(q^{m},q^{n},q).
\end{equation*}
If we set $\ell=1$ then
 \begin{equation*}
\mathfrak{f}_{1,1,1}(q^{m},q^{n},q)=q^{m-n}\mathfrak{f}_{1,1,1}(q^{m},q^{n},q).
\end{equation*}
Because $m\ne n$ and $|q|<1$, we conclude
 \begin{equation*}
\mathfrak{f}_{1,1,1}(q^{m},q^{n},q)=0.\qedhere
\end{equation*}
\end{proof}

\begin{proposition}\label{proposition:symmetry-shift} For $t\in\mathbb{Z}$, we have
\begin{align}
 \mathfrak{f}_{1,7,1}&(q^{2+t},q^{3+t},q)+q^{4+t}\mathfrak{f}_{1,7,1}(q^{6+t},q^{7+t},q)\label{equation:f171-f441}\\
&=\mathfrak{f}_{4,4,1}(-q^{4+t},q^{2+t},q)-q^{-t}\mathfrak{f}_{4,4,1}(-q^{4-t},q^{2-t},q).\notag\\
\mathfrak{f}_{1,5,1}&(q^{2+t},q^{2+t},q)+q^{3+t}\mathfrak{f}_{1,5,1}(q^{5+t},q^{5+t},q)\label{equation:f151-f331}\\
&=\mathfrak{f}_{3,3,1}(-q^{3+t},q^{2+t},q)-q^{-t}\mathfrak{f}_{3,3,1}(-q^{3-t},q^{2-t},q).\notag
\end{align}

\end{proposition}
\begin{proof}[Proof of Proposition \ref{proposition:symmetry-shift}]
The proofs for the identities are similar, so we only do the first.  A change of variables yields
{\allowdisplaybreaks \begin{align*}
\mathfrak{f}_{1,7,1}&(q^{2+t},q^{3+t},q)+q^{4+t}\mathfrak{f}_{1,7,1}(q^{6+t},q^{7+t},q)\\
&=\sum_{\substack{u,v\\u\equiv v \mod 2}}\sg(u,v)(-1)^{\frac{u-v}{2}}q^{\frac{1}{8}u^2+\frac{7}{4}uv+\frac{1}{8}v^2+\frac{3}{4}u+\frac{5}{4}v+\frac{t}{2}(u+v)}\\
&=\Big ( \sum_{\substack{n+j\ge0\\n-j\ge 0}}- \sum_{\substack{n+j < 0\\n-j < 0}}\Big )(-1)^jq^{2n^2+(2+t)n-j(3j+1)/2}\\
&= \sum_{n\ge0}q^{2n^2+(2+t)n}\sum_{j=-n}^{n}(-1)^jq^{-j(3j+1)/2}
 -\sum_{n<0}q^{2n^2+(2+t)n}\sum_{n<j<-n}(-1)^jq^{-j(3j+1)/2}\\
&= \sum_{n\ge0}q^{2n^2+(2+t)n}\sum_{j=-n}^{n}(-1)^jq^{-j(3j+1)/2}\\
&\ \ \ \ \ \ -\sum_{n\ge0}q^{2(-n-1)^2+(2+t)(-n-1)}\sum_{-n-1<j<n+1}(-1)^jq^{-j(3j+1)/2}\\
&= \sum_{n\ge0}q^{2n^2+(2+t)n}\sum_{j=-n}^{n}(-1)^jq^{-j(3j+1)/2}
 -\sum_{n\ge0}q^{2n^2+(2-t)n-t}\sum_{j=-n}^{n}(-1)^jq^{-j(3j+1)/2}\\
&= \sum_{\sg(j)=\sg(n-j)}\sg(j)(-1)^{j}q^{2n^2+(2+t)n-j(3j+1)/2}(1-q^{-2tn-t})\\
&=\mathfrak{f}_{4,4,1}(-q^{4+t},q^{2+t},q)-q^{-t}\mathfrak{f}_{4,4,1}(-q^{4-t},q^{2-t},q),
\end{align*}}%
where the last line follows from the substitution $u=j$ and $v=n-j$.\qedhere
\end{proof}

\begin{proposition}\label{proposition:f441-reduce}  For $m\in\mathbb{Z}$ and $k\in\{0,1,2,3\}$, we have
\begin{align}
\mathfrak{f}_{4,4,1}&(-q^{4+4m+k},q^{2+4m+k},q)-q^{-4m-k}\mathfrak{f}_{4,4,1}(-q^{4-4m-k},q^{2-4m-k},q)
\label{equation:f441-reduce}\\
&=\big ( \mathfrak{f}_{4,4,1}(-q^{4+k},q^{2+k},q)-q^{-k}\mathfrak{f}_{4,4,1}(-q^{4-k},q^{2-k},q) \big )q^{-2m^2-2m-mk}\notag
\end{align}
\end{proposition}

\begin{proposition}\label{proposition:f331-reduce} For $m\in\mathbb{Z}$ and $k\in\{0,1,2\}$, we have
{\allowdisplaybreaks \begin{align}
\mathfrak{f}_{3,3,1}&(-q^{3+3m+k},q^{2+3m+k},q)-q^{-3m-k}\mathfrak{f}_{3,3,1}(-q^{3-3m-k},q^{2-3m-k},q)
\label{equation:f331-reduce}\\
&=\Big ( \mathfrak{f}_{3,3,1}(-q^{3+k},q^{2+k},q)
-q^{-k}\mathfrak{f}_{3,3,1}(-q^{3-k},q^{2-k},q)\Big ) q^{-\frac{3}{2}m^2-\frac{3}{2}m-mk}.\notag
\end{align}}%
\end{proposition}

\begin{lemma}\label{lemma:f441-f331-evals}  If we define
\begin{align}
a_k&:=\mathfrak{f}_{4,4,1}(-q^{4+k},q^{2+k},q)
-q^{-k}\mathfrak{f}_{4,4,1}(-q^{4-k},q^{2-k},q),\\
b_k&:=\mathfrak{f}_{3,3,1}(-q^{3+k},q^{2+k},q)
-q^{-k}\mathfrak{f}_{3,3,1}(-q^{3-k},q^{2-k},q),
\end{align}
then
\begin{equation}
a_0=0, \ a_1 = -q^{-1}J_1^2, \ a_2=0, a_3=q^{-3}J_1^2,\label{equation:f441-evaluations}
\end{equation}
and
\begin{equation}
b_0=0, \ b_1=-q^{-1}J_1^3/J_2, \ b_2=q^{-2}J_1^3/J_2.\label{equation:f331-evaluations}
\end{equation}
\end{lemma}

\begin{proof}[Proof of Proposition \ref{proposition:f441-reduce}] We recall Corollary \ref{corollary:f441-expansion}.  We first show that
{\allowdisplaybreaks \begin{align*}
\mathfrak{h}_{4,4,1}&(-q^{4+m},q^{2+m},q,-1,-1)-q^{-m}\mathfrak{h}_{4,4,1}(-q^{4-m},q^{2-m},q,-1,-1)\\
&=j(-q^{4+m};q^4)m\big ( q,q^3,-1\big )-q^{-m}j(-q^{4-m};q^4)m\big ( q,q^3,-1\big )\\
&=q^{-m}j(-q^{m};q^4)m\big ( q,q^3,-1\big )-q^{-m}j(-q^{m};q^4)m\big ( q,q^3,-1\big )\\
&=0,
\end{align*}}%
where we have used (\ref{equation:j-elliptic}).  As a result,
{\allowdisplaybreaks \begin{align*}
\mathfrak{f}_{4,4,1}&(-q^{4+m},q^{2+m},q)-q^{-m}\mathfrak{f}_{4,4,1}(-q^{4-m},q^{2-m},q)\\
%
&=-\sum_{d=0}^3\frac{q^{3\binom{d+1}{2}}j(q^{5+3d+m};q^4)j(q^{11-3d};q^{12})J_{12}^3j(-q^{3+3d-3m};q^{12})}{\overline{J}_{0,3}\overline{J}_{0,12}j(q^{2-3m};q^{12})j(-q^{1+3d};q^{12})}\\
&\ \ \ \ \ +q^{-m}\sum_{d=0}^3\frac{q^{3\binom{d+1}{2}}j(q^{5+3d-m};q^4)j(q^{11-3d};q^{12})J_{12}^3j(-q^{3+3d+3m};q^{12})}{\overline{J}_{0,3}\overline{J}_{0,12}j(q^{2+3m};q^{12})j(-q^{1+3d};q^{12})}.
\end{align*}
We make the substitution $m\rightarrow 4m+k$, where $k\in\mathbb{Z}$, $0\le k \le 3$:
\begin{align*}
\mathfrak{f}_{4,4,1}&(-q^{4+4m+k},q^{2+4m+k},q)-q^{-4m-k}\mathfrak{f}_{4,4,1}(-q^{4-4m-k},q^{2-4m-k},q)\\
&= -\sum_{d=0}^3\frac{q^{3\binom{d+1}{2}}j(q^{5+3d+4m+k};q^4)j(q^{11-3d};q^{12})J_{12}^3j(-q^{3+3d-12m-3k};q^{12})}{\overline{J}_{0,3}\overline{J}_{0,12}j(q^{2-12m-3k};q^{12})j(-q^{1+3d};q^{12})}\\
&\ \ \ \ \ +q^{-4m-k}\sum_{d=0}^3\frac{q^{3\binom{d+1}{2}}j(q^{5+3d-4m-k};q^4)j(q^{11-3d};q^{12})J_{12}^3j(-q^{3+3d+12m+3k};q^{12})}{\overline{J}_{0,3}\overline{J}_{0,12}j(q^{2+12m+3k};q^{12})j(-q^{1+3d};q^{12})}\\
&= -\sum_{d=0}^3\frac{(-1)^mq^{-m(5+3d+k)}q^{-4\binom{m}{2}}q^{m(3+3d-3k)}q^{-12\binom{-m}{2}}}{(-1)^mq^{m(2-3k)}q^{-12\binom{-m}{2}}} \\
&\ \ \ \ \ \cdot \frac{q^{3\binom{d+1}{2}}j(q^{5+3d+k};q^4)j(q^{11-3d};q^{12})J_{12}^3j(-q^{3+3d-3k};q^{12})}{\overline{J}_{0,3}\overline{J}_{0,12}j(q^{2-3k};q^{12})j(-q^{1+3d};q^{12})}\\
&\ \ \ \ \ +q^{-4m-k}\sum_{d=0}^3\frac{(-1)^mq^{m(5+3d-k)}q^{-4\binom{-m}{2}}q^{-m(3+3d+3k)}q^{-12\binom{m}{2}}}{(-1)^mq^{-m(2+3k)}q^{-12\binom{m}{2}}}\\
&\ \ \ \ \ \cdot \frac{q^{3\binom{d+1}{2}}j(q^{5+3d-k};q^4)j(q^{11-3d};q^{12})J_{12}^3j(-q^{3+3d+3k};q^{12})}{\overline{J}_{0,3}\overline{J}_{0,12}j(q^{2+3k};q^{12})j(-q^{1+3d};q^{12})} \\
&=\big ( \mathfrak{f}_{4,4,1}(-q^{4+k},q^{2+k},q)-q^{-k}\mathfrak{f}_{4,4,1}(-q^{4-k},q^{2-k},q) \big )q^{-2m^2-2m-mk}.\qedhere
\end{align*}} %
\end{proof}

\begin{proof}[Proof of Proposition \ref{proposition:f331-reduce}] We recall Corollary \ref{corollary:f331-expansion}.  We first show that
\begin{align*}
\mathfrak{h}_{3,3,1}&(-q^{3+m},q^{2+m},q)-q^{-m}\mathfrak{h}_{3,3,1}(-q^{3-m},q^{2-m},q)\\
&=j(-q^{3+m};q^3)m\big ( q,q^2,-1\big )-q^{-m}j(-q^{3-m};q^3)m\big ( q,q^2,-1\big )\\
&=q^{-m}j(-q^{m};q^3)m\big ( q,q^2,-1\big )-q^{-m}j(-q^{m};q^3)m\big ( q,q^2,-1\big )\\
&=0,
\end{align*}
where we have used (\ref{equation:j-elliptic}).  As a result, we have
{\allowdisplaybreaks \begin{align*}
\mathfrak{f}_{3,3,1}&(-q^{3+m},q^{2+m},q)-q^{-m}\mathfrak{f}_{3,3,1}(-q^{3-m},q^{2-m},q)\\
%
&=-\sum_{d=0}^{2}
\frac{q^{d(d+1)}j\big (q^{4+2d+m};q^{3}\big )  j\big (q^{5-2d};q^{6}\big ) J_{6}^3j\big (q^{1+2d-2m};q^{6}\big )}
{4\overline{J}_{2,8}\overline{J}_{6,24}j\big (-q^{-2m};q^6\big )j\big (-q^{1+2d};q^6\big )}\\
&\ \ \ \ \ +q^{-m}\sum_{d=0}^{2}
\frac{q^{d(d+1)}j\big (q^{4+2d-m};q^{3}\big )  j\big (q^{5-2d};q^{6}\big ) J_{6}^3j\big (q^{1+2d+2m};q^{6}\big )}
{4\overline{J}_{2,8}\overline{J}_{6,24}j\big (-q^{2m};q^6\big )j\big (-q^{1+2d};q^6\big )}.
\end{align*}}%
We replace $m\rightarrow 3m+k$, where $k\in\mathbb{Z}$ and $k\in\{0,1,2\}$:
{\allowdisplaybreaks \begin{align*}
\mathfrak{f}_{3,3,1}&(-q^{3+3m+k},q^{2+3m+k},q)-q^{-3m-k}\mathfrak{f}_{3,3,1}(-q^{3-3m-k},q^{2-3m-k},q)\\
&= -\sum_{d=0}^{2}
\frac{q^{d(d+1)}j\big (q^{4+2d+3m+k};q^{3}\big )  j\big (q^{5-2d};q^{6}\big ) J_{6}^3j\big (q^{1+2d-6m-2k};q^{6}\big )}
{4\overline{J}_{2,8}\overline{J}_{6,24}j\big (-q^{-6m-2k};q^6\big )j\big (-q^{1+2d};q^6\big )}\\
&\ \ \ \ \ +q^{-3m-k}\sum_{d=0}^{2}
\frac{q^{d(d+1)}j\big (q^{4+2d-3m-k};q^{3}\big )  j\big (q^{5-2d};q^{6}\big ) J_{6}^3j\big (q^{1+2d+6m+2k};q^{6}\big )}
{4\overline{J}_{2,8}\overline{J}_{6,24}j\big (-q^{6m+2k};q^6\big )j\big (-q^{1+2d};q^6\big )}\\
&= -\sum_{d=0}^{2}\frac{(-1)^mq^{m(1+2d-2k)}q^{-6\binom{-m}{2}}(-1)^mq^{-m(4+2d+k)}q^{-3\binom{m}{2}}}
{q^{m(-2k)}q^{-6\binom{-m}{2}}}\\
&\ \ \ \ \ \cdot
\frac{q^{d(d+1)}j\big (q^{4+2d+k};q^{3}\big )  j\big (q^{5-2d};q^{6}\big ) J_{6}^3j\big (q^{1+2d-2k};q^{6}\big )}
{4\overline{J}_{2,8}\overline{J}_{6,24}j\big (-q^{-2k};q^6\big )j\big (-q^{1+2d};q^6\big )}\\
&\ \ \ \ \ +q^{-3m-k}\sum_{d=0}^{2}\frac{(-1)^mq^{-m(1+2d+2k)}q^{-6\binom{m}{2}}(-1)^mq^{m(4+2d-k)}q^{-3\binom{-m}{2}}}
{q^{-m(2k)}q^{-6\binom{m}{2}}}\\
&\ \ \ \ \ \cdot
\frac{q^{d(d+1)}j\big (q^{4+2d-k};q^{3}\big )  j\big (q^{5-2d};q^{6}\big ) J_{6}^3j\big (q^{1+2d+2k};q^{6}\big )}
{4\overline{J}_{2,8}\overline{J}_{6,24}j\big (-q^{2k};q^6\big )j\big (-q^{1+2d};q^6\big )}\\
&=-\sum_{d=0}^{2}q^{-\frac{3}{2}m^2-\frac{3}{2}m-mk}\cdot \frac{q^{d(d+1)}j\big (q^{4+2d+k},q^{3}\big )  j\big (q^{5-2d};q^{6}\big ) J_{6}^3j\big (q^{1+2d-2k};q^{6}\big )}
{4\overline{J}_{2,8}\overline{J}_{6,24}j\big (-q^{-2k};q^6\big )j\big (-q^{1+2d};q^6\big )}\\
&\ \ \ \ \ +q^{-3m-k}\sum_{d=0}^{2}q^{-\frac{3}{2}m^2+\frac{3}{2}m-mk}\cdot 
\frac{q^{d(d+1)}j\big (q^{4+2d-k};q^{3}\big )  j\big (q^{5-2d};q^{6}\big ) J_{6}^3j\big (q^{1+2d+2k};q^{6}\big )}
{4\overline{J}_{2,8}\overline{J}_{6,24}j\big (-q^{2k};q^6\big )j\big (-q^{1+2d};q^6\big )}\\
&=\Big ( \mathfrak{f}_{3,3,1}(-q^{3+k},q^{2+k},q)
-q^{-k}\mathfrak{f}_{3,3,1}(-q^{3-k},q^{2-k},q)\Big ) q^{-\frac{3}{2}m^2-\frac{3}{2}m-mk}.\qedhere
\end{align*}}%
\end{proof}

\begin{proof}[Proof of Lemma \ref{lemma:f441-f331-evals}] The proofs of the evaluations of (\ref{equation:f441-evaluations}) and (\ref{equation:f331-evaluations}) are similar, so we will only do the latter.  That $b_0=0$ is trivial.  Using Corollay \ref{corollary:f331-expansion} and recalling from the proof of Proposition \ref{proposition:f331-reduce} that the $m(x,q,z)$ terms sum to zero, we have
\begin{align*}
b_1
&=-\sum_{d=0}^{2}
\frac{q^{d(d+1)}j\big (q^{5+2d};q^{3}\big )  j\big (q^{5-2d};q^{6}\big ) J_{6}^3j\big (q^{2d-1};q^{6}\big )}
{4\overline{J}_{2,8}\overline{J}_{6,24}j\big (-q^{-2};q^6\big )j\big (-q^{1+2d};q^6\big )}\\
&\ \ \ \ \ +q^{-1}\sum_{d=0}^{2}
\frac{q^{d(d+1)}j\big (q^{3+2d};q^{3}\big )  j\big (q^{5-2d};q^{6}\big ) J_{6}^3j\big (q^{3+2d};q^{6}\big )}
{4\overline{J}_{2,8}\overline{J}_{6,24}j\big (-q^{2};q^6\big )j\big (-q^{1+2d};q^6\big )}\\
%
&=
-2q^{-1}\frac{J_1 J_{1,6}^2 J_{6}^3}
{4\overline{J}_{2,8}\overline{J}_{6,24}\overline{J}_{2,6}\overline{J}_{1,6}}
-2q^{-1}\frac{J_1  J_{3,6}J_{6}^3J_{1,6}}
{4\overline{J}_{2,8}\overline{J}_{6,24}j\overline{J}_{2,6}\overline{J}_{3,6}},
\end{align*}
where we have used (\ref{equation:j-elliptic}) and simplified.  Continuing with the calculation we have
\begin{align*}
b_1
&=-2q^{-1}\cdot \frac{J_1 J_{1,6} J_{6}^3}{4\overline{J}_{2,8}\overline{J}_{6,24}\overline{J}_{2,6}}
\cdot \Big [ \frac{J_{1,6}\overline{J}_{3,6}
+J_{3,6}\overline{J}_{1,6}}{\overline{J}_{1,6}\overline{J}_{3,6}}\Big ] \\
&=-2q^{-1}\cdot \frac{J_1 J_{1,6} J_{6}^3}{4\overline{J}_{2,8}\overline{J}_{6,24}\overline{J}_{2,6}}
\cdot \Big [ \frac{2J_{4,12}^2}{\overline{J}_{1,6}\overline{J}_{3,6}}\Big ],
\end{align*}
where we used (\ref{equation:H1Thm1.2B}).   Elementary product rearrangements yield the result.  For the final evaluation we recall (\ref{equation:f151-f331})
{\allowdisplaybreaks \begin{align*}
b_2
&=f_{1,5,1}(q^{4},q^{4},q)+q^{5}f_{1,5,1}(q^{7},q^{7},q)\\
&=-q^{-1}f_{1,5,1}(q^{3},q^{3},q)-q^{-2}f_{1,5,1}(1,1,q)\\
&=-q^{-1}f_{1,5,1}(q^{3},q^{3},q)-q^{5}f_{1,5,1}(q^6,q^6,q)\\
&=-q^{-1}b_1,
\end{align*}}%
where for the second equality we used (\ref{equation:fabc-flip}), for the third equality we used (\ref{equation:f-shift}) with $(\ell,k)=(1,1)$, and for the last equality we used (\ref{equation:f151-f331}).
\end{proof}

\section{Motivation}\label{section:preview}

We explore the use of the heuristic in determining the modularity of triple-sums.  First we recall how the heuristic is useful in studying Hecke-type double-sums, see also \cite[Section 3]{HM}, \cite{M2014}.  We have \cite[$(1.15)$]{H2}
\begin{equation}
\sum_{\sg(r)=\sg(s)}\sg(r)c_{r,s}= \sum_{\sg(r)=\sg(s)}\sg(r)c_{r+\ell,s+k}+\sum_{r=0}^{\ell-1}\sum_{s\in \mathbb{Z}}c_{r,s}+\sum_{s=0}^{k-1}\sum_{r\in \mathbb{Z}}c_{r,s}.\label{equation:H2eq1.15}
\end{equation}
Letting $\ell, k \rightarrow \infty$ in (\ref{equation:H2eq1.15}) we have
\begin{equation}
\sum_{\sg(r)=\sg(s)}\sg(r)c_{r,s}\sim\sum_{r\ge0}\sum_{s\in \mathbb{Z}}c_{r,s}+\sum_{s\ge0}\sum_{r\in \mathbb{Z}}c_{r,s},\label{equation:double-sum-heuristic}
\end{equation}
where `$\sim$' again means `mod theta'.   This relation is useful in determining expansions of Hecke-type double-sums \cite{HM}.  As an example, consider
{\allowdisplaybreaks \begin{align*}
\sum_{\sg(r)=\sg(s)}&\sg(r)(-1)^{r+s}x^ry^sq^{\binom{r}{2}+2rs+\binom{s}{2}}\\
&\sim \sum_{r=0}^{\infty}\sum_{s\in\mathbb{Z}}(-1)^{r+s}x^ry^sq^{\binom{r}{2}+2rs+\binom{s}{2}}
+ \sum_{s=0}^{\infty}\sum_{r\in\mathbb{Z}}(-1)^{r+s}x^ry^sq^{\binom{r}{2}+2rs+\binom{s}{2}}\\
&\sim \sum_{r=0}^{\infty}(-1)^rx^rq^{\binom{r}{2}}\sum_{s\in\mathbb{Z}}(-1)^{s}y^sq^{2rs+\binom{s}{2}}
+ \sum_{s=0}^{\infty}(-1)^sq^{\binom{s}{2}}y^s\sum_{r\in\mathbb{Z}}(-1)^{r}x^rq^{\binom{r}{2}+2rs}\\
&\sim \sum_{r=0}^{\infty}(-1)^rx^rq^{\binom{r}{2}}j(yq^{2r};q)
+ \sum_{s=0}^{\infty}(-1)^sq^{\binom{s}{2}}y^sj(xq^{2s};q)\\
&\sim j(y;q)\sum_{r=0}^{\infty}(-1)^rx^rq^{\binom{r}{2}}y^{-2r}q^{-\binom{2r}{2}}
+ j(x;q)\sum_{s=0}^{\infty}(-1)^sq^{\binom{s}{2}}y^sx^{-2s}q^{-\binom{2s}{2}}\\
&\sim j(y;q)\sum_{r=0}^{\infty}(-1)^rq^{2r}x^ry^{-2r}q^{-3\binom{r+1}{2}}
+ j(x;q)\sum_{s=0}^{\infty}(-1)^sq^{2s}y^sx^{-2s}q^{-3\binom{s+1}{2}},
\end{align*}}%
where in the last line we have used the elliptic transformation property (\ref{equation:j-elliptic}).  Recalling the heuristic (\ref{equation:mxqz-heuristic}), we have
\begin{align*}
\sum_{\sg(r)=\sg(s)}&\sg(r)(-1)^{r+s}x^ry^sq^{\binom{r}{2}+2rs+\binom{s}{2}}
\sim j(y;q)m(q^2x/y^2,q^3,*)
+ j(x;q)m(q^2y/x^2,q^3,*).
\end{align*}
The `$\sim$' becomes an equality upon specialising the `$*$' \ 's and adding an appropriate theta function (\ref{equation:f121}).

\section{Triple-sum relations}\label{section:obtain-the-triple}
We obtain an identity for triple-sums analogous to (\ref{equation:H2eq1.15}).  Without loss of generality, we assume $R,S,T\ge 0$.  We write 
{\allowdisplaybreaks \begin{align*}
\sum_{\sg(r)=\sg(s)=\sg(t)}&c_{r,s,t}- \sum_{\sg(r)=\sg(s)=\sg(t)}c_{r+R,s+S,t+T}\\
&=\sum_{\sg(r)=\sg(s)=\sg(t)}c_{r,s,t}- \sum_{\sg(r-R)=\sg(s-S)=\sg(t-T)}c_{r,s,t}\\
&=\Big ( \sum_{r,s,t\ge0} +\sum_{r,s,t<0}\Big )c_{r,s,t}
-\Big ( \sum_{r-R,s-S,t-T\ge0} +\sum_{r-R,s-S,t-T<0}\Big )c_{r,s,t}\\
&=\Big ( \sum_{r,s,t\ge0} -\sum_{r-R,s-S,t-T\ge0}\Big )c_{r,s,t}
-\Big (\sum_{r-R,s-S,t-T<0}-\sum_{r,s,t<0} \Big )c_{r,s,t}.
\end{align*}}%
We rewrite some terms to have
{\allowdisplaybreaks \begin{align*}
\sum_{\sg(r)=\sg(s)=\sg(t)}&c_{r,s,t} - \sum_{\sg(r)=\sg(s)=\sg(t)}c_{r+R,s+S,t+T}\\
&=\Big [ \sum_{\substack{0\le r <R \\ s,t\ge 0}} 
+\sum_{\substack{r\ge R \\ s\ge S \\0\le t <T}}
+ \sum_{\substack{r\ge R \\ 0\le s<S\\t\ge T  }}
+\sum_{\substack{r\ge R \\ 0\le s<S\\0\le t < T  }}\Big ]c_{r,s,t}\\
& \ \ \ \ \ - \Big [ \sum_{\substack{0\le r<R \\ s<S\\ t< T}} 
+\sum_{\substack{r< 0 \\ s< 0 \\0 \le t < T}}
+ \sum_{\substack{r< 0 \\ 0\le s<S \\t < 0  }}
+\sum_{\substack{r< 0 \\ 0\le s<S\\0\le t < T  }}\Big ]c_{r,s,t}\\
&=\Big [ \sum_{\substack{0\le r <R \\ s,t\ge 0}} +\Big ( \sum_{\substack{r\ge 0 \\ s\ge 0 \\0\le t <T}}-\sum_{\substack{ 0\le r <R \\ s\ge S \\0\le t <T}}-\sum_{\substack{r\ge R \\ 0\le s<S \\0\le t <T}}-\sum_{\substack{0\le r<R \\ 0\le s<S \\0\le t <T}} \Big )\\
&\ \ \ \ \ \ \ \ \ \ + \Big ( \sum_{\substack{r\ge 0 \\ 0\le s<S\\t\ge 0  }}- \sum_{\substack{0\le r< R \\ 0\le s<S\\t\ge T  }} - \sum_{\substack{r\ge R \\ 0\le s<S\\ 0\le t<T  }} - \sum_{\substack{0\le r < R \\ 0\le s<S\\0\le t < T  }}\Big ) +\sum_{\substack{r\ge R \\ 0\le s<S\\0\le t < T  }}\Big ] c_{r,s,t}\\
& \ \ \ \ \ - \Big [ \Big ( \sum_{\substack{0\le r<R \\ s<0\\ t< 0}}+\sum_{\substack{0\le r<R \\ 0 \le s<S \\ t< 0}}+ \sum_{\substack{0\le r<R \\ s<0\\ 0\le t<T }}+ \sum_{\substack{0\le r<R \\ 0\le s<S\\ 0\le t<T}} \Big )
+\sum_{\substack{r< 0 \\ s< 0 \\0 \le t < T}}+ \sum_{\substack{r< 0 \\ 0\le s<S \\t < 0  }}+\sum_{\substack{r< 0 \\ 0\le s<S\\0\le t < T  }}\Big ] c_{r,s,t}.
\end{align*}}%
Rearranging terms, we have
{\allowdisplaybreaks \begin{align}
\sum_{\sg(r)=\sg(s)=\sg(t)}&c_{r,s,t}- \sum_{\sg(r)=\sg(s)=\sg(t)}c_{r+R,s+S,t+T}\label{equation:generic-shift}\\
&=\sum_{r=0}^{R-1}\sum_{\sg(s)=\sg(t)}\sg(s) c_{r,s,t} +\sum_{s=0}^{S-1}\sum_{\sg(r)=\sg(t)}\sg(r) c_{r,s,t} +\sum_{t=0}^{T-1}\sum_{\sg(r)=\sg(s)}\sg(r) c_{r,s,t} \notag\\
&\ \ \ \ \ -\sum_{r=0}^{R-1}\sum_{t=0}^{T-1}\sum_{s\in \mathbb{Z}}c_{r,s,t}-\sum_{s=0}^{S-1}\sum_{t=0}^{T-1}\sum_{r\in \mathbb{Z}}c_{r,s,t} 
-\sum_{r=0}^{R-1}\sum_{s=0}^{S-1}\sum_{t\in \mathbb{Z}}c_{r,s,t}.\notag
\end{align}}%
Letting $R,S,T \rightarrow \infty$, and discarding the $c_{r+R,s+S,t+T}$ term as in (\ref{equation:double-sum-heuristic}), we arrive at 
{\allowdisplaybreaks \begin{align}
\sum_{\sg(r)=\sg(s)=\sg(t)}&c_{r,s,t}\label{equation:HeckeTriple-id}\\
&\sim \sum_{r=0}^{\infty}\sum_{\sg(s)=\sg(t)}\sg(s)c_{r,s,t}+\sum_{s=0}^{\infty}\sum_{\sg(r)=\sg(t)}\sg(r)c_{r,s,t}
+\sum_{t=0}^{\infty}\sum_{\sg(r)=\sg(s)}\sg(r)c_{r,s,t}\notag\\
&\ \ \ \ \ - \sum_{r=0}^{\infty}\sum_{t=0}^{\infty}\sum_{s\in \mathbb{Z}}c_{r,s,t}
- \sum_{r=0}^{\infty}\sum_{s=0}^{\infty}\sum_{t\in \mathbb{Z}}c_{r,s,t}
- \sum_{s=0}^{\infty}\sum_{t=0}^{\infty}\sum_{r\in \mathbb{Z}}c_{r,s,t},\notag
\end{align}}%
where `$\sim$' is `mod theta' in determining the modularity of triple sums of the form (\ref{equation:triple-sum-def}).

We finish with two propositions whose proofs are straightforward.
\begin{proposition} We have
{\allowdisplaybreaks \begin{align}
\mathfrak{g}_{a,b,c,d,e,f}&(x,y,z,q)\label{equation:g-shift}\\
&=(-1)^{R+S+T}x^{R}y^{S}z^{T}q^{a\binom{R}{2}+bRS+c\binom{S}{2}+dRT+eST+f\binom{T}{2}}\notag \\
& \ \ \ \ \ \cdot \mathfrak{g}_{a,b,c,d,e,f}(q^{aR+bS+dT}x,q^{bR+cS+eT}y,q^{dR+eS+fT}z,q)\notag\\
& \ \ \ \ \ + \sum_{r=0}^{R-1}(-1)^{r}x^rq^{a\binom{r}{2}}\mathfrak{f}_{c,e,f}(q^{br}y,q^{dr}z,q)
+ \sum_{s=0}^{S-1}(-1)^{s}y^sq^{c\binom{s}{2}}\mathfrak{f}_{a,d,f}(q^{bs}x,q^{es}z,q)\notag\\
& \ \ \ \ \ + \sum_{t=0}^{T-1}(-1)^{t}z^tq^{f\binom{t}{2}}\mathfrak{f}_{a,b,c}(q^{dt}x,q^{et}y,q)\notag\\
& \ \ \ \ \ -\sum_{r=0}^{R-1} (-1)^{r}x^rq^{a\binom{r}{2}}
\sum_{t=0}^{T-1}(-1)^{t}z^tq^{drt+f\binom{t}{2}}j(q^{br+et}y;q^{c})\notag\\
& \ \ \ \ \ -\sum_{s=0}^{S-1} (-1)^{s}y^sq^{c\binom{s}{2}}
\sum_{t=0}^{T-1}(-1)^{t}z^tq^{est+f\binom{t}{2}}j(q^{bs+dt}x;q^{a})\notag\\
& \ \ \ \ \ -\sum_{r=0}^{R-1} (-1)^{r}x^rq^{a\binom{r}{2}}
\sum_{s=0}^{S-1}(-1)^{s}y^sq^{brs+c\binom{t}{2}}j(q^{dr+es}z;q^{f}).\notag
\end{align}}%
\end{proposition}

\begin{proposition} We have
\begin{equation}
\mathfrak{g}_{a,b,c,d,e,f}(x,y,z,q)=-\frac{q^{a+b+c+d+e+f}}{xyz}
\mathfrak{g}_{a,b,c,d,e,f}(\frac{q^{2a+b+d}}{x},\frac{q^{b+2c+e}}{y},\frac{q^{d+e+2f}}{z},q).\label{equation:triple-flip}
\end{equation}
\end{proposition}

\section{A heuristic guide to mock theta function identities}\label{section:Zwegers}
The left-hand sides of identities (\ref{equation:chi0-triple}) and (\ref{equation:chi1-triple}) may be viewed as specializations of
\begin{equation}
\mathfrak{g}_{1,2,1,2,2,1}(x,y,z,q):=\Big ( \sum_{k,l,m\ge 0}+\sum_{k,l,m<0}\Big )(-1)^{k+l+m}
q^{\binom{k}{2}+2kl+\binom{l}{2}+2km+2lm+\binom{m}{2}}x^ky^lz^m.\label{equation:triple-general}
\end{equation}
In particular, Identities (\ref{equation:chi0-triple}) and (\ref{equation:chi1-triple}) may be rewritten
\begin{gather}
2-\frac{1}{J_1^2}\cdot \mathfrak{g}_{1,2,1,2,2,1}(q,q,q,q)=\chi_0(q),\label{equation:goal1}\\
\frac{1}{J_1^2}\cdot \mathfrak{g}_{1,2,1,2,2,1}(q^2,q^2,q^2,q)=\chi_1(q).\label{equation:goal2}
\end{gather}
Using  (\ref{equation:5th-chi0(q)}) and (\ref{equation:5th-chi1(q)}), we may rewrite the above identities as
{\allowdisplaybreaks \begin{gather}
2-\frac{1}{J_1^2}\cdot \mathfrak{g}_{1,2,1,2,2,1}(q,q,q,q)
=2-3m(q^7,q^{15},q^9)-3q^{-1}m(q^2,q^{15},q^4)+\frac{2J_5^2J_{2,5}}{J_{1,5}^2},\label{equation:goal1-final}\\
\frac{1}{J_1^2}\cdot \mathfrak{g}_{1,2,1,2,2,1}(q^2,q^2,q^2,q)
=-3q^{-1}m(q^4,q^{15},q^3)-3q^{-2}m(q,q^{15},q^2)-\frac{2J_5^2J_{1,5}}{J_{2,5}^2}.\label{equation:goal2-final}
\end{gather}}%
In this section, we demonstrate how heuristic methods take us from the left-hand sides of (\ref{equation:goal1-final}) and (\ref{equation:goal2-final}) to their respective right-hand sides up to a theta function.

Let us apply the heuristic to (\ref{equation:triple-general}).  We recall the (slightly rewritten) general form (\ref{equation:HeckeTriple-id}):
{\allowdisplaybreaks \begin{align}
&\sum_{\sg(k)=\sg(l)=\sg(m)}c_{k,l,m}\label{equation:HeckeTriple-2}\\
&\ \ \ \ \ \sim \sum_{k=0}^{\infty}\sum_{\sg(l)=\sg(m)}\sg(l)c_{k,l,m}+\sum_{l=0}^{\infty}\sum_{\sg(k)=\sg(m)}\sg(k)c_{k,l,m}
+\sum_{m=0}^{\infty}\sum_{\sg(k)=\sg(l)}\sg(k)c_{k,l,m}\notag\\
&\ \ \ \ \ \ \ \ \ \ - \sum_{k=0}^{\infty}\sum_{m=0}^{\infty}\sum_{l\in \mathbb{Z}}c_{k,l,m}
- \sum_{k=0}^{\infty}\sum_{l=0}^{\infty}\sum_{m\in \mathbb{Z}}c_{k,l,m}
- \sum_{l=0}^{\infty}\sum_{m=0}^{\infty}\sum_{k\in \mathbb{Z}}c_{k,l,m}.\notag
\end{align}}%
We have two types of summands on the right-hand side of (\ref{equation:HeckeTriple-2}).  We consider the first type of summand:
{\allowdisplaybreaks \begin{align}
\sum_{k\ge 0}&\sum_{l\ge 0}\sum_{m\in \mathbb{Z}}(-1)^{k+l+m}
q^{\binom{k}{2}+\binom{l}{2}+\binom{m}{2}+2kl+2km+2lm}x^ky^lz^m\label{equation:first-type}\\
&\sim \sum_{k\ge 0}(-1)^kx^kq^{\binom{k}{2}}\sum_{l\ge 0}(-1)^{l}y^lq^{\binom{l}{2}+2kl}\sum_{m\in \mathbb{Z}}(-1)^{m}
q^{\binom{m}{2}+2km+2lm}z^m\notag \\
&\sim \sum_{k\ge 0}(-1)^kx^kq^{\binom{k}{2}}\sum_{l\ge 0}(-1)^{l}y^lq^{\binom{l}{2}+2kl}j(zq^{2k+2l};q).\notag 
\end{align}}%
We note that if we set $z$ to be $q$ or $q^2$ then $j(z;q)=0$, so we can ignore the contributions from the second line of (\ref{equation:HeckeTriple-2}).  We consider the second type of summand:
{\allowdisplaybreaks \begin{align}
\sum_{m\ge 0}&\sum_{k,l}\sg(k,l)(-1)^{k+l+m}
q^{\binom{k}{2}+\binom{l}{2}+\binom{m}{2}+2kl+2km+2lm}x^ky^lz^m\label{equation:second-type}\\
&\sim\sum_{m\ge 0}(-1)^mq^{\binom{m}{2}}z^m\sum_{k,l}\sg(k,l)(-1)^{k+l}
q^{\binom{k}{2}+\binom{l}{2}+2kl+2km+2lm}x^ky^l\notag\\
&\sim \sum_{m\ge 0}(-1)^mq^{\binom{m}{2}}z^m\sum_{k,l}\sg(k,l)(-1)^{k+l}
q^{\binom{k}{2}+2kl+\binom{l}{2}}(q^{2m}x)^k(q^{2m}y)^l\notag\\
&\sim \sum_{m\ge 0}(-1)^mq^{\binom{m}{2}}z^m\mathfrak{f}_{1,2,1}(q^{2m}x,q^{2m}y,q).\notag
\end{align}}%

We rewrite the last line of (\ref{equation:second-type}).  We recall (\ref{equation:f121}):
{\allowdisplaybreaks \begin{align}
\mathfrak{ f}_{1,2,1}&(x,y,q)\label{equation:n1p1}\\
&=j(y;q)m\big (\tfrac{q^2x}{y^2},q^3,-1\big )+j(x;q)m\big (\tfrac{q^2y}{x^2},q^3,-1\big )
- \frac{yJ_{3}^3j(-x/y;q)j(q^2xy;q^3)}
{\overline{J}_{0,3}j(-qy^2/x,-qx^2/y;q^3)}.\notag
\end{align}}%
Let us consider the double-sum:
\begin{equation}
\mathfrak{f}_{1,2,1}(q^{2m}x,q^{2m}y,q).
\end{equation}
The first thing we note is that if $x$ and $y$ are either $q$ or $q^2$, then theta coefficients of the $m(x,q,z)$ functions in (\ref{equation:n1p1}) are both zero, and both $m(x,q,z)$ functions are defined.  So we only need to consider the theta function from the right-hand side of (\ref{equation:n1p1}) which after $x\rightarrow q^{2m}x$ and $y\rightarrow q^{2m}y$ is
\begin{equation}
\mathfrak{f}_{1,2,1}(q^{2m}x,q^{2m}y,q)
= - \frac{yq^{2m}J_{3}^3j(-x/y;q)j(q^{2+4m}xy;q^3)}{\overline{J}_{0,3}j(-q^{1+2m}y^2/x,-q^{1+2m}x^2/y;q^3)}.
\end{equation}
In order to use the elliptic transformation property (\ref{equation:j-elliptic}), we will have to consider $m$ mod three, so let $m \rightarrow 3m+a$.  We first consider the case $a=0$
{\allowdisplaybreaks \begin{align*}
-&\sum_{m\ge 0}(-1)^mq^{\binom{3m}{2}}z^{3m} \frac{yq^{6m}J_{3}^3j(-x/y;q)j(q^{2+12m}xy;q^3)}{\overline{J}_{0,3}j(-q^{1+6m}y^2/x,-q^{1+6m}x^2/y;q^3)}\\
&\sim -\frac{yJ_{3}^3j(-x/y;q)j(q^2xy;q^3)}{\overline{J}_{0,3}j(-qy^2/x,-qx^2/y;q^3)}\\
& \ \ \ \ \ \ \ \ \ \ \cdot \sum_{m\ge 0}(-1)^mq^{\binom{3m}{2}}z^{3m} q^{6m}
\frac{(q^2xy)^{-4m}q^{-3\binom{4m}{2}}}{(-qy^2/x)^{-2m}(-qx^2/y)^{-2m}q^{-6\binom{2m}{2}}}\\
&\sim -\frac{yJ_{3}^3j(-x/y;q)j(q^2xy;q^3)}{\overline{J}_{0,3}j(-qy^2/x,-qx^2/y;q^3)}
\sum_{m\ge 0}(-1)^m\big (\frac{z^3}{x^2y^2}\big )^{m}q^{-15m^2/2+m/2}\\
&\sim -\frac{yJ_{3}^3j(-x/y;q)j(q^2xy;q^3)}{\overline{J}_{0,3}j(-qy^2/x,-qx^2/y;q^3)}
 m\Big (\frac{q^8z^3}{x^2y^2},q^{15},*\Big ).
\end{align*}}%

We consider $m \rightarrow 3m+1$.  Then
{\allowdisplaybreaks \begin{align*}
&\sum_{m\ge 0}(-1)^mq^{\binom{3m+1}{2}}z^{3m+1} \frac{yq^{6m+2}J_{3}^3j(-x/y;q)j(q^{6+12m}xy;q^3)}{\overline{J}_{0,3}j(-q^{3+6m}y^2/x,-q^{3+6m}x^2/y;q^3)}\\
&\sim \frac{yJ_{3}^3j(-x/y;q)j(xy;q^3)}{\overline{J}_{0,3}j(-y^2/x,-x^2/y;q^3)}\\
& \ \ \ \ \ \ \ \ \ \ \cdot \sum_{m\ge 0}(-1)^mq^{\binom{3m+1}{2}}z^{3m+1} q^{6m+2}
\frac{(xy)^{-4m-2}q^{-3\binom{4m+2}{2}}}{(-y^2/x)^{-2m-1}(-x^2/y)^{-2m-1}q^{-6\binom{2m+1}{2}}}\\
&\sim \frac{yJ_{3}^3j(-x/y;q)j(xy;q^3)}{\overline{J}_{0,3}j(-y^2/x,-x^2/y;q^3)}
\frac{z}{xy}
\sum_{m\ge 0}(-1)^m\big (\frac{z^3}{x^2y^2}\big )^{m}q^{-15m^2/2-9m/2-1}\\
&\sim \frac{yJ_{3}^3j(-x/y;q)j(xy;q^3)}{\overline{J}_{0,3}j(-y^2/x,-x^2/y;q^3)}
\cdot \frac{z}{xyq}\cdot m\Big (\frac{q^3z^3}{x^2y^2},q^{15},*\Big ).
\end{align*}}%

We let $m\rightarrow 3m+2$.  We have
{\allowdisplaybreaks \begin{align*}
-&\sum_{m\ge 0}(-1)^{3m+2}q^{\binom{3m+2}{2}}z^{3m+2}
 \frac{yq^{6m+4}J_{3}^3j(-x/y;q)j(q^{12m+10}xy;q^3)}{\overline{J}_{0,3}j(-q^{6m+5}y^2/x,-q^{6m+5}x^2/y;q^3)}\\
 &\sim - \frac{yJ_{3}^3j(-x/y;q)j(qxy;q^3)}{\overline{J}_{0,3}j(-q^{2}y^2/x,-q^{2}x^2/y;q^3)}\\
& \ \ \ \ \ \ \ \ \ \  \cdot \sum_{m\ge 0}(-1)^{m}q^{\binom{3m+2}{2}}z^{3m+2}q^{6m+4}
\frac{(-1)^{4m+3}q^{-3\binom{4m+3}{2}}(qxy)^{-(4m+3)}}
{q^{-6\binom{2m+1}{2}}(q^2x^2/y)^{-(2m+1)}(q^2y^2/x)^{-(2m+1)}}\\
 &\sim  \frac{yJ_{3}^3j(-x/y;q)j(qxy;q^3)}{\overline{J}_{0,3}j(-q^{2}y^2/x,-q^{2}x^2/y;q^3)}
  \cdot \frac{z^2}{x^2y^2} \cdot \sum_{m\ge 0}(-1)^{m}\Big ( \frac{z^3}{x^2y^2}\Big )^{m} 
q^{-15m^2/2-19m/2-3}\\
 &\sim  \frac{yJ_{3}^3j(-x/y;q)j(qxy;q^3)}{\overline{J}_{0,3}j(-q^{2}y^2/x,-q^{2}x^2/y;q^3)}
  \cdot \frac{z^2}{x^2y^2q^3} \cdot 
m\Big (\frac{z^3}{q^2x^2y^2},q^{15},*\Big ).
\end{align*}}%
For $x,y,z$ integral powers of $q$, we suspect that we have
{\allowdisplaybreaks \begin{align}
\mathfrak{g}_{1,2,1,2,2,1}&(x,y,z,q)\\
&\sim
-y\cdot \frac{J_{3}^3j(-x/y;q)j(q^2xy;q^3)}{\overline{J}_{0,3}j(-qy^2/x,-qx^2/y;q^3)}
 \cdot m\Big (\frac{q^8z^3}{x^2y^2},q^{15},*\Big ) + idem(z;x,y)\notag \\
& \ \ \ \ \ +  \frac{z}{xq} \cdot \frac{J_{3}^3j(-x/y;q)j(xy;q^3)}{\overline{J}_{0,3}j(-y^2/x,-x^2/y;q^3)}
\cdot m\Big (\frac{q^3z^3}{x^2y^2},q^{15},*\Big )+idem(z;x,y)\notag \\
&\ \ \ \ \ +   \frac{z^2}{x^2yq^3} \cdot \frac{J_{3}^3j(-x/y;q)j(qxy;q^3)}{\overline{J}_{0,3}j(-q^{2}y^2/x,-q^{2}x^2/y;q^3)}
 \cdot  m\Big (\frac{z^3}{q^2x^2y^2},q^{15},*\Big )+idem(z;x,y),\notag 
\end{align}}%
where `$\sim$' means modulo a theta function and $idem(z;x;y)$ means the preceding term is repeated twice--once with $z$ and $x$ swapped and once with $z$ and $y$ swapped.

Let us try the values for $\chi_0(q)$, i.e. $x=y=z=q$.  In the third summand on the the right-hand side, we have a $j(q^3;q^3)=0$ in the numerator of the quotient of theta functions, so only the first two summands contribute.  Using (\ref{equation:chi0-triple}) gives
{\allowdisplaybreaks \begin{align*}
2-\chi_0(q)&=\frac{1}{J_1^2}\cdot \mathfrak{g}_{1,2,1,2,2,1}(q,q,q,q)\\
&\  \sim -3q \frac{J_3^3j(-1;q)j(q^4;q^3)}{J_1^2\overline{J}_{0,3}j(-q^2;q^3)^2}m(q^7,q^{15},*)
+3q^{-1}\frac{J_3^3j(-1;q)j(q^2;q^3)}{J_1^2\overline{J}_{0,3}j(-q;q^3)^2}m(q^2,q^{15},*)\\
&\  \sim 3 \frac{J_3^3\overline{J}_{0,1}J_1}{J_1^2\overline{J}_{0,3}\overline{J}_{1,3}^2}m(q^7,q^{15},*)
+3q^{-1}\frac{J_3^3\overline{J}_{0,1}J_1}{J_1^2\overline{J}_{0,3}\overline{J}_{1,3}^2}m(q^2,q^{15},*)\\
&\  \sim 3m(q^7,q^{15},*) + 3q^{-1}m(q^{2},q^{15},*),
\end{align*}}%
where we have used (\ref{equation:j-elliptic}) and product rearrangements.  This agrees with (\ref{equation:5th-chi0(q)}) up to a theta function.

Let us the values for $\chi_0(q)$, i.e. $x=y=z=q^2$.   In the first summand on the the right-hand side, we have a $j(q^6;q^3)=0$ in the numerator of the quotient of theta functions, so only the second and third summands contribute.  Using (\ref{equation:chi1-triple}), we arrive at
\begin{align*}
\chi_1(q)=\frac{1}{J_1^2}\cdot \mathfrak{g}_{1,2,1,2,2,1}(q^2,q^2,q^2,q) &\sim -3q^{-2}m(q,q^{15},*)-3q^{-5}m(q^{-4},q^{15},*)\\
&\sim -3q^{-2}m(q,q^{15},*)-3q^{-1}m(q^{4},q^{15},*),
\end{align*}
where we have used (\ref{equation:mxqz-flip}).  This agrees with (\ref{equation:5th-chi1(q)}) up to a theta function.

\section{A heuristic guide to false theta function identities}\label{section:KimLovejoy}  We recall a theorem of Kim and Lovejoy: 
\begin{theorem}\cite[Theorem $1.1$]{KL} We have
{\allowdisplaybreaks \begin{align}
\sum_{n\ge 0}\frac{(q)_{2n}q^n}{(aq,q/a)_n}&=(1-a)\sum_{\substack{r,s\ge 0 \\ r\equiv s \pmod 2}}(-1)^ra^{\frac{r+s}{2}}q^{\frac{3}{2}+rs+\frac{1}{2}r+s}\notag \\
& \ \ \ \ \ +\frac{(q)_{\infty}}{(aq,q/a)_{\infty}}\sum_{r\ge 0}(-1)^ra^{2r+1}q^{3r(r+1)/2},\\
\sum_{n=0}^{\infty}\frac{(q;q^2)_n(q)_nq^{n}}{(aq,q/a)_{n}}&=(1-a)\sum_{\substack{r,s\ge 0\\ r\equiv s \pmod 2}}(-1)^ra^{\frac{r+s}{2}}q^{rs+\frac{1}{2}r+\frac{1}{2}s}\notag\\
&\ \ \ \ \  +\frac{(q)_{\infty}}{(aq,q/a,-q)_{\infty}}\sum_{r\ge 0}(-1)^ra^{3r+1}q^{3r^2+2r}(1+aq^{2r+1}).
\end{align}}%
\end{theorem}
We collect the $a=1$ specializations in the following corollary:
\begin{corollary} \label{corollary:KL-Thm1.1-spec} We have
\begin{align}
\sum_{n = 0}^{\infty}\frac{(q;q)_{2n}q^n}{(q;q)_n^2}&=\frac{1}{J_1}\sum_{r=0}^{\infty}(-1)^rq^{3r(r+1)/2},\label{equation:KL-id-A}\\
\sum_{n = 0}^{\infty}\frac{(q;q^2)_nq^n}{(q;q)_n}&=\frac{1}{J_2}\sum_{r = 0}^{\infty}(-1)^rq^{3r^2+2r}(1+q^{2r+1}).\label{equation:KL-id-B}
\end{align}
\end{corollary}
We point out that (\ref{equation:KL-id-B}) is the $a=-1$ specialization of \cite[(5.2)]{Warn}.  In \cite{KL}, Kim and Lovejoy also show how to write \cite[Theorem $1.1$]{KL} in terms of Hecke-type triple-sums.
\begin{corollary}
 \label{corollary:KL-prop5.1-5.2}  \cite[Propositions $5.1$, $5.2$]{KL} We have
{\allowdisplaybreaks \begin{align}
\sum_{n\ge 0}\frac{(q)_{2n}q^n}{(aq,q/a)_n}&=\frac{1}{(q,aq,q/a;q)_{\infty}}\Big ( \mathfrak{g}_{1,7,1,1,1,1}(aq^2,q^3,q,q)
\\
&\ \ \ \ \ \ \ \ \ \ \ \ \ \ \ \ \ \ \ \ \ \ \ \ \ +q^4\mathfrak{g}_{1,7,1,1,1,1}(aq^6,q^7,q^2,q)\Big ),\notag\\
\sum_{n\ge 0}\frac{(q;q^2)_{n}(q)_nq^n}{(aq,q/a)_n}&=\frac{1}{(q,aq,q/a;q)_{\infty}}\Big ( \mathfrak{g}_{1,5,1,1,1,1}(aq^2,q^2,q,q)\\
&\ \ \ \ \ \ \ \ \ \ \ \ \ \ \ \ \ \ \ \ \ \ \ \ \ 
+q^3\mathfrak{g}_{1,5,1,1,1,1}(aq^5,q^5,q^2,q)\Big ).\notag 
\end{align}}%
\end{corollary}
In this section, we will use the triple-sum relation (\ref{equation:HeckeTriple-id}) and the Hecke-type double-sum expansions of \cite{HM} in order to guide us from the $a=1$ specialization of Corollary \ref{corollary:KL-prop5.1-5.2} to the false theta functions of Corollary \ref{corollary:KL-Thm1.1-spec}.

\subsection{On the triple-sum for the false theta function Identity (\ref{equation:KL-id-A})} 
We apply our heuristic methods to
\begin{equation*}
\mathfrak{g}_{1,7,1,1,1,1}(q^2,q^3,q,q)+q^4\mathfrak{g}_{1,7,1,1,1,1}(q^6,q^7,q^2,q).
\end{equation*}
We consider the contributions of $\mathfrak{g}_{1,7,1,1,1,1}(q^2,q^3,q,q)$ to the top row of (\ref{equation:generic-shift}).  We have
\begin{align*}
\sum_{r=0}^{R-1}\sum_{\sg(s)=\sg(t)}\sg(s) c_{r,s,t}
&=\sum_{r=0}^{R-1}(-1)^{r}q^{2r}q^{\binom{r}{2}}\sum_{s,t}\sg(s,t)(-1)^{s+t}q^{3s}q^tq^{7rs+\binom{s}{2}+rt+st+\binom{t}{2}}\\
&=\sum_{r= 0}^{R-1}(-1)^rq^{2r}q^{\binom{r}{2}}\mathfrak{f}_{1,1,1}(q^{3+7r},q^{1+r},q)\\
&= 0,
\end{align*}
where for the last equality we used Proposition \ref{proposition:f111-zero}.  Similarly, we have
\begin{equation*}
\sum_{s=0}^{S-1}\sum_{\sg(r)=\sg(t)}\sg(r) c_{r,s,t}
=\sum_{s= 0}^{S-1}(-1)^sq^{3s}q^{\binom{s}{2}}\mathfrak{f}_{1,1,1}(q^{2+7s},q^{1+s},q)=0.
\end{equation*}
For the final piece in the top row of (\ref{equation:generic-shift}), we have
{\allowdisplaybreaks \begin{align*}
\sum_{t=0}^{T-1}\sum_{\sg(r)=\sg(s)}\sg(r) c_{r,s,t} 
&= \sum_{t= 0}^{T-1}(-1)^tq^tq^{\binom{t}{2}}\mathfrak{f}_{1,7,1}(q^{2+t},q^{3+t},q).
\end{align*}}%

We consider the bottom row of (\ref{equation:generic-shift}).  We have
{\allowdisplaybreaks \begin{align*}
\sum_{r=0}^{R-1}\sum_{s= 0}^{S-1}\sum_{t\in\mathbb{Z}}c_{r,s,t}
&= \sum_{r=0}^{R-1}\sum_{s= 0}^{S-1}(-1)^{r+s}q^{2r}q^{3s}q^{\binom{r}{2}+7rs+\binom{s}{2}}\sum_{t\in\mathbb{Z}}(-1)^{t}q^{(1+r+s)t}q^{\binom{t}{2}}\\
&= \sum_{r=0}^{R-1}\sum_{s= 0}^{S-1}(-1)^{r+s}q^{2r}q^{3s}q^{\binom{r}{2}+7rs+\binom{s}{2}}j(q^{1+r+s};q)\\
&= 0.
\end{align*}}%
Similarly, we have
 \begin{align*}
\sum_{r=0}^{R-1}\sum_{t=0}^{T-1}\sum_{t\in\mathbb{Z}}c_{r,s,t}
&= \sum_{r=0}^{R-1}\sum_{t=0}^{T-1}(-1)^{r+t}q^{2r}q^{t}q^{\binom{r}{2}+rt+\binom{t}{2}}j(q^{3+t+7r};q)
= 0,\\
\sum_{s= 0}^{S-1}\sum_{t= 0}^{T-1}\sum_{r\in\mathbb{Z}}c_{r,s,t}
&= \sum_{s= 0}^{S-1}\sum_{t= 0}^{T-1}(-1)^{s+t}q^{3s}q^tq^{\binom{s}{2}+st+\binom{r}{2}}j(q^{2+7s+t};q)
=0.
\end{align*}

We proceed to work on the second sum $\mathfrak{g}_{1,7,1,1,1,1}(q^6,q^7,q^2,q)$.  We consider the contributions to the top row of (\ref{equation:generic-shift}).  Arguing as above we have

{\allowdisplaybreaks \begin{align*}
\sum_{r=0}^{R-1}\sum_{\sg(s)=\sg(t)}\sg(s) c_{r,s,t}
&=\sum_{r=0}^{R-1}(-1)^rq^{6r}q^{\binom{r}{2}}\mathfrak{f}_{1,1,1}(q^{7+7r},q^{2+r},q)
= 0,\\
\sum_{s=0}^{S-1}\sum_{\sg(r)=\sg(t)}\sg(r) c_{r,s,t}
&=\sum_{s=0}^{S-1}(-1)^sq^{7s}q^{\binom{s}{2}}\mathfrak{f}_{1,1,1}(q^{6+7s},q^{2+s},q)
= 0,\\
\sum_{t=0}^{T-1}\sum_{\sg(r)=\sg(s)}\sg(r) c_{r,s,t}
&=\sum_{t= 0}^{T-1}(-1)^tq^{2t}q^{\binom{t}{2}}\mathfrak{f}_{1,7,1}(q^{6+t},q^{7+t},q).
\end{align*}}%
For the bottom row, we have
{\allowdisplaybreaks \begin{align*}
\sum_{r=0}^{R-1}\sum_{s=0}^{S-1}\sum_{t\in\mathbb{Z}}c_{r,s,t}
&= \sum_{r=0}^{R-1}\sum_{s=0}^{S-1}(-1)^{r+s}q^{6r}q^{7s}q^{\binom{r}{2}+7rs+\binom{s}{2}}j(q^{2+r+s};q)
= 0,\\
\sum_{r=0}^{R-1}\sum_{t=0}^{T-1}\sum_{t\in\mathbb{Z}}c_{r,s,t}
&= \sum_{r=0}^{R-1}\sum_{t=0}^{T-1}(-1)^{r+t}q^{6r}q^{2t}q^{\binom{r}{2}+rt+\binom{t}{2}}j(q^{6+7r+t};q)
= 0,\\
\sum_{s= 0}^{S-1}\sum_{t= 0}^{T-1}\sum_{r\in\mathbb{Z}}c_{r,s,t}
&= \sum_{s= 0}^{S-1}\sum_{t= 0}^{T-1}(-1)^{s+t}q^{7s}q^{2t}q^{\binom{s}{2}+st+\binom{t}{2}}j(q^{6+7s+t};q)
= 0.
\end{align*}}%

Combining terms and letting $R,S,T\rightarrow \infty$ as in (\ref{equation:HeckeTriple-id}), gives
\begin{align*}
\sum_{n\ge 0}\frac{(q)_{2n}q^n}{(q)_n^2}
& \sim \frac{1}{J_1^3}\Big [ \sum_{t=0}^{\infty}(-1)^tq^{\binom{t+1}{2}}
\Big ( \mathfrak{f}_{1,7,1}(q^{2+t},q^{3+t},q)+q^{4+t}\mathfrak{f}_{1,7,1}(q^{6+t},q^{7+t},q)\Big )\Big ].
\end{align*}
Using Identity (\ref{equation:f171-f441}) gives
\begin{equation*}
\sum_{n\ge 0}\frac{(q)_{2n}q^n}{(q)_n^2}\sim \frac{1}{J_1^3}\Big [ \sum_{m=0}^{\infty}(-1)^mq^{\binom{m+1}{2}}
\Big ( \mathfrak{f}_{4,4,1}(-q^{4+m},q^{2+m},q)-q^{-m}\mathfrak{f}_{4,4,1}(-q^{4-m},q^{2-m},q)\Big )\Big ].
\end{equation*}
Making the substitution $m \rightarrow 4m+k$, $k\in\{0,1,2,3\}$, and using Proposition \ref{proposition:f441-reduce} yields
{\allowdisplaybreaks \begin{align*}
\sum_{n\ge 0}&\frac{(q)_{2n}q^n}{(q)_n^2}\\
&\sim \frac{1}{J_1^3}\sum_{k=0}^3\Big ( \mathfrak{f}_{4,4,1}(-q^{4+k},q^{2+k},q)
-q^{-k}\mathfrak{f}_{4,4,1}(-q^{4-k},q^{2-k},q) \Big ) \\
&\ \ \ \ \ \cdot \Big [ \sum_{m=0}^{\infty}(-1)^{4m+k}q^{\binom{4m+k+1}{2}}q^{-2m^2-2m-mk}\Big ]\\
&\sim \frac{1}{J_1^3}\sum_{k=0}^3\Big [ (-1)^kq^{\binom{k+1}{2}}\Big ( \mathfrak{f}_{4,4,1}(-q^{4+k},q^{2+k},q)
-q^{-k}\mathfrak{f}_{4,4,1}(-q^{4-k},q^{2-k},q) \Big )
\cdot   \sum_{m=0}^{\infty}q^{6m^2+3mk}\Big ].
\end{align*}}%
We recall the evaluations found in (\ref{equation:f441-evaluations}).  If we define
\begin{align*}
a_k:=\mathfrak{f}_{4,4,1}(-q^{4+k},q^{2+k},q)-q^{-k}\mathfrak{f}_{4,4,1}(-q^{4-k},q^{2-k},q),
\end{align*}
then $a_0=0$, $a_1=-q^{-1}J_1^2$, $a_2=0$, and $a_3=q^{-3}J_1^2$.  Finally, we arrive at Identity (\ref{equation:KL-id-A}):
{\allowdisplaybreaks \begin{align*}
\sum_{n\ge 0}\frac{(q)_{2n}q^n}{(q)_n^2}&\sim \frac{1}{J_1^3}\Big [ 0-q\cdot (-q^{-1}J_1^2)\sum_{m=0}^{\infty}q^{6m^2+3m}+0-q^6(q^{-3}J_1^2)\sum_{m=0}^{\infty}q^{6m^2+9m}\Big ]\\
&\sim \frac{1}{J_1}\Big [ \sum_{m=0}^{\infty}q^{6m^2+3m}-q^3\sum_{m=0}^{\infty}q^{6m^2+9m}\Big ]\\
&\sim \frac{1}{J_1}\sum_{m=0}(-1)^mq^{3m(m+1)/2}.
\end{align*}}%

\subsection{On the triple-sum for the false theta function Identity (\ref{equation:KL-id-B})}
We recall our triple-sum
\begin{align*}
\mathfrak{g}_{1,5,1,1,1,1}(q^2,q^2,q,q)&+q^3\mathfrak{g}_{1,5,1,1,1,1}(q^5,q^5,q^2,q).
\end{align*}

We consider the contributions of $\mathfrak{g}_{1,5,1,1,1,1}(q^2,q^2,q,q)$ to the top row of (\ref{equation:generic-shift}). Arguing as in the previous subsection, we have
{\allowdisplaybreaks \begin{align*}
\sum_{r=0}^{R-1}\sum_{\sg(s)=\sg(t)}\sg(s)c_{r,s,t}
&= \sum_{r=0}^{R-1}(-1)^rq^{2r}q^{\binom{r}{2}}\mathfrak{f}_{1,1,1}(q^{2+5r},q^{1+r},q)
= 0,\\
\sum_{s=0}^{S-1}\sum_{\sg(r)=\sg(t)}\sg(r)c_{r,s,t}
&=\sum_{s=0}^{S-1}(-1)^sq^{2s}q^{\binom{s}{2}}\mathfrak{f}_{1,1,1}(q^{2+5s},q^{1+s},q)
=0,\\
\sum_{t=0}^{T-1}\sum_{\sg(r)=\sg(s)}\sg(r)c_{r,s,t}
&= \sum_{t=0}^{T-1}(-1)^tq^tq^{\binom{t}{2}}\mathfrak{f}_{1,5,1}(q^{2+t},q^{2+t},q).
\end{align*}}%

For the bottom row, we have
{\allowdisplaybreaks \begin{align*}
 \sum_{r=0}^{R-1}\sum_{t=0}^{T-1}\sum_{s\in \mathbb{Z}}c_{r,s,t}
&= \sum_{r=0}^{R-1}\sum_{t=0}^{T-1}(-1)^{r+t}q^{2r}q^tq^{\binom{r}{2}+rt+\binom{t}{2}}j(q^{2+5r+t};q)
= 0,\\
\sum_{s=0}^{S-1}\sum_{t=0}^{T-1}\sum_{r\in \mathbb{Z}}c_{r,s,t} 
&= \sum_{s=0}^{S-1}\sum_{t=0}^{T-1}(-1)^{st}q^{2s}q^{t}q^{\binom{s}{2}+st+\binom{t}{2}}j(q^{2+5s+t};q)
=0,\\
\sum_{r=0}^{R-1}\sum_{s=0}^{S-1}\sum_{t\in \mathbb{Z}}c_{r,s,t}
&= \sum_{r=0}^{R-1}\sum_{s=0}^{S-1}(-1)^{r+s}q^{2r}q^{2s}q^{\binom{r}{2}+5rs+\binom{s}{2}}j(q^{1+r+s};q)
=0.
\end{align*}}%

We consider the contributions of $\mathfrak{g}_{1,5,1,1,1,1}(q^5,q^5,q^2,q)$ to the top row of (\ref{equation:generic-shift}).  Computing the summands yields
{\allowdisplaybreaks \begin{align*}
 \sum_{r=0}^{R-1}\sum_{\sg(s)=\sg(t)}\sg(s) c_{r,s,t}
&= \sum_{r=0}^{R-1}(-1)^rq^{5r}q^{\binom{r}{2}}\mathfrak{f}_{1,1,1}(q^{5+5r},q^{2+r},q)
= 0,\\
 \sum_{s=0}^{S-1}\sum_{\sg(r)=\sg(t)}\sg(r) c_{r,s,t}
&= \sum_{s=0}^{S-1}(-1)^sq^{5s}q^{\binom{s}{2}}\mathfrak{f}_{1,1,1}(q^{5+5s},q^{2+s},q)
=0,\\
 \sum_{t=0}^{T-1}\sum_{\sg(r)=\sg(s)}\sg(r) c_{r,s,t} 
&=\sum_{t=0}^{T-1}(-1)^tq^{2t}q^{\binom{t}{2}}\mathfrak{f}_{1,5,1}(q^{5+t},q^{5+t},q).
\end{align*}}%
For the bottom row, we have as before
{\allowdisplaybreaks \begin{align*}
\sum_{r=0}^{R-1}\sum_{s=0}^{S-1}\sum_{t\in \mathbb{Z}}c_{r,s,t}
&= \sum_{r=0}^{R-1}\sum_{s=0}^{S-1}(-1)^{r+s}q^{5r}q^{5s}q^{\binom{r}{2}+5rs+\binom{s}{2}}j(q^{2+r+s};q)
= 0,\\
 \sum_{r=0}^{R-1}\sum_{t=0}^{T-1}\sum_{s\in \mathbb{Z}}c_{r,s,t}
& = \sum_{r=0}^{R-1}\sum_{t=0}^{T-1}(-1)^{r+t}q^{5r}q^{2t}q^{\binom{r}{2}+rt+\binom{t}{2}}j(q^{5+5r+t};q)
 =0,\\
\sum_{s=0}^{S-1}\sum_{t=0}^{T-1}\sum_{r\in \mathbb{Z}}c_{r,s,t} 
&= \sum_{s=0}^{S-1}\sum_{t=0}^{T-1}(-1)^{s+t}q^{5s}q^{2t}q^{\binom{s}{2}+st+\binom{t}{2}}j(q^{5+5s+t};q)
=0.
\end{align*}}%

Combining terms and letting $R,S,T\rightarrow \infty$ as in (\ref{equation:HeckeTriple-id}), gives
\begin{align*}
\sum_{n\ge 0}\frac{(q;q^2)_{n}q^n}{(q)_n}
& \sim \frac{1}{J_1^3}\Big [ \sum_{t=0}^{\infty}(-1)^tq^{\binom{t+1}{2}}
\Big ( \mathfrak{f}_{1,5,1}(q^{2+t},q^{2+t},q)+q^{3+t}\mathfrak{f}_{1,5,1}(q^{5+t},q^{5+t},q)\Big )\Big ].
\end{align*}
Using Identity (\ref{equation:f151-f331}) gives
\begin{align*}
\sum_{n\ge 0}&\frac{(q)_{2n}q^n}{(q)_n^2}\sim \frac{1}{J_1^3}\Big [ \sum_{m=0}^{\infty}(-1)^mq^{\binom{m+1}{2}}
\Big ( \mathfrak{f}_{3,3,1}(-q^{3+m},q^{2+m},q)-q^{-m}\mathfrak{f}_{3,3,1}(-q^{3-m},q^{2-m},q)\Big )\Big ].
\end{align*}
Making the substitution $m\rightarrow 3m+k$, $k\in\{0,1,2\}$, and using Proposition \ref{proposition:f331-reduce} yields
{\allowdisplaybreaks \begin{align*}
\sum_{n=0}^{\infty}\frac{(q;q^2)_nq^n}{(q)_n}
%
%
&\sim \frac{1}{J_1^3}\sum_{k=0}^{2}  \Big (  f_{3,3,1}(-q^{3+k},q^{2+k},q)
-q^{-k}f_{3,3,1}(-q^{3-k},q^{2-k},q)\Big )\\
& \ \ \ \ \ \cdot \Big [\sum_{m=0}^{\infty}(-1)^{3m+k}q^{\binom{3m+k+1}{2}}  q^{-\frac{3}{2}m^2-\frac{3}{2}m-mk} \Big ]\\
&\sim \frac{1}{J_1^3}\sum_{k=0}^{2}  \Big [ (-1)^kq^{\binom{k+1}{2}}\Big (  f_{3,3,1}(-q^{3+k},q^{2+k},q)
-q^{-k}f_{3,3,1}(-q^{3-k},q^{2-k},q)\Big ) \\
&\ \ \ \ \ \ \ \ \ \ \cdot \sum_{m=0}^{\infty}(-1)^{m}q^{3m^2+2mk} \Big ].
\end{align*}}%
We recall the evaluations found in (\ref{equation:f331-evaluations}).  If we define
{\allowdisplaybreaks \begin{align*}
b_k:= \mathfrak{f}_{3,3,1}(-q^{3+k},q^{2+k},q)-q^{-k}\mathfrak{f}_{3,3,1}(-q^{3-k},q^{2-k},q),
\end{align*}
then $b_0=0$, $b_1=-q^{-1}{J_1^3}/{J_2}$, and $b_2=q^{-2}{J_1^3}/{J_2}$.  Finally, we arrive at Identity (\ref{equation:KL-id-B}):
{\allowdisplaybreaks \begin{align*}
\sum_{n=0}^{\infty}&\frac{(q;q^2)_nq^n}{(q)_n}\\
&\sim \frac{1}{J_1^3}\Big [  0-q(-q^{-1}J_1^3/J_2)\sum_{m=0}^{\infty}(-1)^{m}q^{3m^2+2m}+q^3(q^{-2}J_1^3/J_2)\sum_{m=0}^{\infty}(-1)^{m}q^{3m^2+4m}\Big ]\\
&\sim \frac{1}{J_2} \sum_{m=0}^{\infty}(-1)^{m}q^{3m^2+2m}(1+q^{2m+1}).
\end{align*}}%

\section{Proof of Corollary \ref{corollary:newIds}}\label{section:exploration}
We present some corollaries to identities (\ref{equation:chi0-triple}) and (\ref{equation:chi1-triple}).  Identity (\ref{equation:newId-1}) follows from (\ref{equation:triple-flip}).  For Identity (\ref{equation:newId-2}), we employ (\ref{equation:g-shift}) with $(R,S,T)=(0,0,1)$ to obtain
\begin{align*}
\mathfrak{g}_{1,2,1,2,2,1}(q,q,q^2,q)&=-q^2\mathfrak{g}_{1,2,1,2,2,1}(q^3,q^3,q^3,q)+\mathfrak{f}_{1,2,1}(q,q,q)
=\mathfrak{f}_{1,2,1}(q,q,q)
=J_{1}^2,
\end{align*}
where we used (\ref{equation:newId-1}) for the second equality and (\ref{equation:f121}) for the third equality.   For Identity (\ref{equation:newId-3}), we use (\ref{equation:g-shift}) with $(R,S,T)=(0,1,1)$ to have
\begin{align*}
\mathfrak{g}_{1,2,1,2,2,1}(q,q^2,q^2,q)
&=q^{6}\mathfrak{g}_{1,2,1,2,2,1}(q^5,q^5,q^5,q)+2\mathfrak{f}_{1,2,1}(q,q^2,q)\\
&=-\mathfrak{g}_{1,2,1,2,2,1}(q,q,q,q)+2\mathfrak{f}_{1,2,1}(q,q^2,q)\\
&=-\mathfrak{g}_{1,2,1,2,2,1}(q,q,q,q)+2J_{1}^2,
\end{align*}
where the second equality follows from (\ref{equation:triple-flip}) and the third equality from (\ref{equation:f121}).  Comparing the last equality with (\ref{equation:chi0-triple}) gives the result.  for Identity (\ref{equation:newId2-1}), we again use (\ref{equation:g-shift}) but with $(R,S,T)=(0,0,2)$ to obtain
\begin{align*}
\mathfrak{g}_{1,2,1,2,2,1}(q,q,q^3,q)
&=q^{7}\mathfrak{g}_{1,2,1,2,2,1}(q^5,q^5,q^5,q)+\mathfrak{f}_{1,2,1}(q,q,q)-q^3\mathfrak{f}_{1,2,1}(q^3,q^3,q)\\
&=-q\mathfrak{g}_{1,2,1,2,2,1}(q,q,q,q)+\mathfrak{f}_{1,2,1}(q,q,q)-q^3\mathfrak{f}_{1,2,1}(q^3,q^3,q)\\
&=-q\mathfrak{g}_{1,2,1,2,2,1}(q,q,q,q)+\mathfrak{f}_{1,2,1}(q,q,q)+q\mathfrak{f}_{1,2,1}(q,q,q)\\
&=qJ_1^2(\chi_0(q)-2)+J_1^2+qJ_1^2\\
&=qJ_1^2\chi_0(q)+J_1^2-qJ_1^2,
\end{align*}
where the second equality follows from (\ref{equation:triple-flip}), the third equality from (\ref{equation:fabc-flip}), the fourth equality from (\ref{equation:f121}), and the last equality from  (\ref{equation:chi0-triple}).  For the final Identity (\ref{equation:newId2-2}) we use (\ref{equation:g-shift}) with $(R,S,T)=(-1,0,1)$ to have
\begin{align*}
\mathfrak{g}_{1,2,1,2,2,1}&(q,q^2,q^3,q)\\
&=q\mathfrak{g}_{1,2,1,2,2,1}(q^2,q^2,q^2,q)-q^{-1}\mathfrak{f}_{1,2,1}(1,q,q)+\mathfrak{f}_{1,2,1}(q,q,q)
+q^{-1}j(1;q).
\end{align*}
Using (\ref{equation:f121}) twice gives
\begin{equation*}
\mathfrak{g}_{1,2,1,2,2,1}(q,q^2,q^3,q)
=q\mathfrak{g}_{1,2,1,2,2,1}(q^2,q^2,q^2,q)+J_1^2
=qJ_1^2\chi_1(q)+J_1^2,
\end{equation*}
where the last equality followed from (\ref{equation:chi1-triple}) and (\ref{equation:f121}).


\section{Obtaining Identity (\ref{equation:g131331-conjecture})}\label{section:conjecture}
We consider
\begin{equation}
\mathfrak{g}_{1,3,1,3,3,1}(q,q,q,q).
\end{equation}
We recall the specialization $(n,p)=(1,2)$ of \cite[Proposition $8.1$]{HM} where $\ell=1$:
\begin{align}
\mathfrak{f}_{1,3,1}(x,y,q)&=j(y;q)m\Big( -\frac{q^5x}{y^3},q^8,\frac{q^2y}{x}\Big  )
+j(x;q)m\Big( -\frac{q^5y}{x^3},q^8,\frac{x}{q^2y}\Big  )\\
&\ \ \ \ \ +\frac{q^{5}x^2yJ_{2,4}J_{8,16}j(q^7xy;q^8)j(q^{22}x^2y^2;q^{16})}{j(-q^{5}x^2;q^8)j(-q^{9}y^2;q^8)}.\notag
\end{align}
Following the reasoning in Section \ref{section:Zwegers}, we only need to consider terms of the form
\begin{equation}
\sum_{m\ge0}(-1)^{m}q^{\binom{m}{2}}z^m\mathfrak{f}_{1,3,1}(q^{3m}x,q^{3m}y,q).\label{equation:f131-sum}
\end{equation}
Because the theta coefficients of the two Appell--Lerch terms vanish and because the Appell--Lerch terms are defined, we can write the above as
\begin{equation}
\mathfrak{f}_{1,3,1}(q^{3m}x,q^{3m}y,q)
=
\frac{q^{5+9m}x^2yJ_{2,4}J_{8,16}j(q^{7+6m}xy;q^8)j(q^{22+12m}x^2y^2;q^{16})}
{j(-q^{5+6m}x^2;q^8)j(-q^{9+6m}y^2;q^8)}.
\end{equation}

We will need to consider $m\rightarrow 4m+a$, $a\in \{0,1,2,3\}$.  In general, (\ref{equation:f131-sum}) reads
\begin{align*}
\sum_{m\ge0}&(-1)^{a}q^{\binom{4m+a}{2}}z^{4m+a}
\frac{q^{5+36m+9a}x^2yJ_{2,4}J_{8,16}j(q^{7+24m+6a}xy;q^8)j(q^{22+48m+12a}x^2y^2;q^{16})}
{j(-q^{5+24m+6a}x^2;q^8)j(-q^{9+24m+6a}y^2;q^8)}.
\end{align*}
For the case $a=0$
\begin{align*}
\sum_{m\ge0}&q^{\binom{4m}{2}}z^{4m}
\frac{q^{5+36m}x^2yJ_{2,4}J_{8,16}j(q^{7+24m}xy;q^8)j(q^{22+48m}x^2y^2;q^{16})}
{j(-q^{5+24m}x^2;q^8)j(-q^{9+24m}y^2;q^8)}\\
&\sim
\sum_{m\ge0}q^{\binom{4m}{2}}z^{4m}
\frac{q^{5+36m}x^2yJ_{2,4}J_{8,16}j(q^{8\cdot 3m}q^{7}xy;q^8)j(q^{16\cdot (3m+1)}q^{6}x^2y^2;q^{16})}
{j(-q^{8\cdot 3m}q^{5}x^2;q^8)j(-q^{8\cdot (3m+1)}qy^2;q^8)}\\
&\sim
\sum_{m\ge0}q^{\binom{4m}{2}}z^{4m}
\frac{q^{5+36m}x^2yJ_{2,4}J_{8,16}j(q^{7}xy;q^8)j(q^{6}x^2y^2;q^{16})}
{j(-q^5x^2;q^8)j(-qy^2;q^8)}\\
&\ \ \ \ \ 
\cdot\frac{(-1)^{3m}q^{-8\binom{3m}{2}}(q^7xy)^{-3m}(-1)^{3m+1}q^{-16\binom{3m+1}{2}}(q^6x^2y^2)^{-(3m+1)}}
{q^{-8\binom{3m}{2}}(q^5x^2)^{-3m}   q^{-8\binom{3m+1}{2}}(qy^2)^{-(3m+1)} }\\
&\sim
-y\cdot \frac{J_{2,4}J_{8,16}j(q^{7}xy;q^8)j(q^{6}x^2y^2;q^{16})}{j(-q^5x^2;q^8)j(-qy^2;q^8)}
\sum_{m\ge0}\Big ( \frac{z^4}{x^3y^3}\Big )^{m}q^{-28m^2+m} \\
&\sim
-y\cdot \frac{J_{2,4}J_{8,16}j(q^{7}xy;q^8)j(q^{6}x^2y^2;q^{16})}{j(-q^5x^2;q^8)j(-qy^2;q^8)}
\sum_{m\ge0}(-1)^m\Big ( -\frac{q^{29}z^4}{x^3y^3}\Big )^{m}q^{-56\binom{m+1}{2}} \\
&\sim
-y\cdot \frac{J_{2,4}J_{8,16}j(q^{7}xy;q^8)j(q^{6}x^2y^2;q^{16})}{j(-q^5x^2;q^8)j(-qy^2;q^8)}
m\Big ( -\frac{q^{29}z^4}{x^3y^3},q^{56} , * \Big ). 
\end{align*}

For the case $a=1$
\begin{align*}
-\sum_{m\ge0}&q^{\binom{4m+1}{2}}z^{4m+1}
\frac{q^{14+36m}x^2yJ_{2,4}J_{8,16}j(q^{13+24m}xy;q^8)j(q^{34+48m}x^2y^2;q^{16})}
{j(-q^{11+24m}x^2;q^8)j(-q^{15+24m}y^2;q^8)}\\
&\sim
-\sum_{m\ge0}q^{\binom{4m+1}{2}}z^{4m+1}
\frac{q^{14+36m}x^2yJ_{2,4}J_{8,16}j(q^{8\cdot (3m+1)}q^{5}xy;q^8)j(q^{16\cdot (3m+2)}q^{2}x^2y^2;q^{16})}
{j(-q^{8\cdot (3m+1)}q^{3}x^2;q^8)j(-q^{8\cdot (3m+1)}q^7y^2;q^8)}\\
&\sim
-\sum_{m\ge0}q^{\binom{4m+1}{2}}z^{4m+1}
\frac{q^{14+36m}x^2yJ_{2,4}J_{8,16}j(q^{5}xy;q^8)j(q^{2}x^2y^2;q^{16})}
{j(-q^3x^2;q^8)j(-q^7y^2;q^8)}\\
&\ \ \ \ \ 
\cdot\frac{(-1)^{3m+1}q^{-8\binom{3m+1}{2}}(q^5xy)^{-(3m+1)}(-1)^{3m+2}
q^{-16\binom{3m+2}{2}}(q^2x^2y^2)^{-(3m+2)}}
{q^{-8\binom{3m+1}{2}}(q^3x^2)^{-(3m+1)}   q^{-8\binom{3m+1}{2}}(q^7y^2)^{-(3m+1)} }\\
&\sim
\frac{z}{xy^2}\cdot \frac{J_{2,4}J_{8,16}j(q^{5}xy;q^8)j(q^{2}x^2y^2;q^{16})}{j(-q^3x^2;q^8)j(-q^7y^2;q^8)}
\sum_{m\ge0}\Big ( \frac{z^4}{x^3y^3}\Big )^{m}q^{-28m^2-13m-1} \\
&\sim
\frac{z}{xy^2q}\cdot \frac{J_{2,4}J_{8,16}j(q^{5}xy;q^8)j(q^{2}x^2y^2;q^{16})}{j(-q^3x^2;q^8)j(-q^7y^2;q^8)}
\sum_{m\ge0}(-1)^m\Big ( -\frac{q^{15}z^4}{x^3y^3}\Big )^{m}q^{-56\binom{m+1}{2}} \\
&\sim
\frac{z}{xy^2q}\cdot \frac{J_{2,4}J_{8,16}j(q^{5}xy;q^8)j(q^{2}x^2y^2;q^{16})}{j(-q^3x^2;q^8)j(-q^7y^2;q^8)}
m\Big ( -\frac{q^{15}z^4}{x^3y^3},q^{56} , * \Big ). 
\end{align*}

For the case $a=2$
\begin{align*}
\sum_{m\ge0}&q^{\binom{4m+2}{2}}z^{4m+2}
\frac{q^{23+36m}x^2yJ_{2,4}J_{8,16}j(q^{19+24m}xy;q^8)j(q^{46+48m}x^2y^2;q^{16})}
{j(-q^{17+24m}x^2;q^8)j(-q^{21+24m}y^2;q^8)}\\
&\sim
\sum_{m\ge0}q^{\binom{4m+2}{2}}z^{4m+2}
\frac{q^{23+36m}x^2yJ_{2,4}J_{8,16}j(q^{8\cdot (3m+2)}q^{3}xy;q^8)j(q^{16\cdot (3m+2)}q^{14}x^2y^2;q^{16})}
{j(-q^{8\cdot (3m+2)}qx^2;q^8)j(-q^{8\cdot (3m+2)}q^5y^2;q^8)}\\
&\sim
\sum_{m\ge0}q^{\binom{4m+2}{2}}z^{4m+2}
\frac{q^{23+36m}x^2yJ_{2,4}J_{8,16}j(q^{3}xy;q^8)j(q^{14}x^2y^2;q^{16})}
{j(-qx^2;q^8)j(-q^5y^2;q^8)}\\
&\ \ \ \ \ 
\cdot\frac{(-1)^{3m+2}q^{-8\binom{3m+2}{2}}(q^3xy)^{-(3m+2)}(-1)^{3m+2}
q^{-16\binom{3m+2}{2}}(q^{14}x^2y^2)^{-(3m+2)}}
{q^{-8\binom{3m+2}{2}}(qx^2)^{-(3m+2)}   q^{-8\binom{3m+2}{2}}(q^5y^2)^{-(3m+2)} }\\
&\sim
\frac{z^2}{y}\cdot \frac{J_{2,4}J_{8,16}j(q^{3}xy;q^8)j(q^{14}x^2y^2;q^{16})}{j(-qx^2;q^8)j(-q^5y^2;q^8)}
\sum_{m\ge0}\Big ( \frac{z^4}{x^3y^3}\Big )^{m}q^{-28m^2-27m-6} \\
&\sim
\frac{z^2}{yq^6}\cdot \frac{J_{2,4}J_{8,16}j(q^{3}xy;q^8)j(q^{14}x^2y^2;q^{16})}{j(-qx^2;q^8)j(-q^5y^2;q^8)}
\sum_{m\ge0}(-1)^{m}\Big ( -\frac{qz^4}{x^3y^3}\Big )^{m}q^{-56\binom{m+1}{2}} \\
&\sim
\frac{z^2}{yq^6}\cdot \frac{J_{2,4}J_{8,16}j(q^3xy;q^8)j(q^{14}x^2y^2;q^{16})}{j(-qx^2;q^8)j(-q^5y^2;q^8)}
m\Big ( -\frac{qz^4}{x^3y^3},q^{56} , * \Big ). 
\end{align*}

For the case $a=3$
\begin{align*}
-\sum_{m\ge0}&q^{\binom{4m+3}{2}}z^{4m+3}
\frac{q^{32+36m}x^2yJ_{2,4}J_{8,16}j(q^{25+24m}xy;q^8)j(q^{58+48m}x^2y^2;q^{16})}
{j(-q^{23+24m}x^2;q^8)j(-q^{27+24m}y^2;q^8)}\\
&\sim
-\sum_{m\ge0}q^{\binom{4m+3}{2}}z^{4m+3}
\frac{q^{32+36m}x^2yJ_{2,4}J_{8,16}j(q^{8\cdot (3m+3)}qxy;q^8)j(q^{16\cdot (3m+3)}q^{10}x^2y^2;q^{16})}
{j(-q^{8\cdot (3m+2)}q^7x^2;q^8)j(-q^{8\cdot (3m+3)}q^3y^2;q^8)}\\
&\sim
-\sum_{m\ge0}q^{\binom{4m+3}{2}}z^{4m+3}
\frac{q^{32+36m}x^2yJ_{2,4}J_{8,16}j(qxy;q^8)j(q^{10}x^2y^2;q^{16})}
{j(-q^7x^2;q^8)j(-q^3y^2;q^8)}\\
&\ \ \ \ \ 
\cdot\frac{(-1)^{3m+3}q^{-8\binom{3m+3}{2}}(qxy)^{-(3m+3)}(-1)^{3m+3}
q^{-16\binom{3m+3}{2}}(q^{10}x^2y^2)^{-(3m+3)}}
{q^{-8\binom{3m+2}{2}}(q^7x^2)^{-(3m+2)}   q^{-8\binom{3m+3}{2}}(q^{3}y^2)^{-(3m+3)} }\\
&\sim
-\frac{z^3}{x^3y^2}\cdot \frac{J_{2,4}J_{8,16}j(qxy;q^8)j(q^{10}x^2y^2;q^{16})}{j(-q^7x^2;q^8)j(-q^3y^2;q^8)}
\sum_{m\ge0}\Big ( \frac{z^4}{x^3y^3}\Big )^{m}q^{-28m^2-41m-15} \\
&\sim
-\frac{z^3}{x^3y^2q^{15}}\cdot  \frac{J_{2,4}J_{8,16}j(qxy;q^8)j(q^{10}x^2y^2;q^{16})}{j(-q^7x^2;q^8)j(-q^3y^2;q^8)}
\sum_{m\ge0}(-1)^m\Big ( -\frac{z^4}{q^{13}x^3y^3}\Big )^{m}q^{-56\binom{m+1}{2}} \\
&\sim
-\frac{z^3}{x^3y^2q^{15}}\cdot  \frac{J_{2,4}J_{8,16}j(qxy;q^8)j(q^{10}x^2y^2;q^{16})}{j(-q^7x^2;q^8)j(-q^3y^2;q^8)}
m\Big ( -\frac{z^4}{q^{13}x^3y^3},q^{56} , * \Big ). 
\end{align*}

Modulo a theta function, our heuristic methods suggest
{\allowdisplaybreaks \begin{align*}
&\mathfrak{g}_{1,3,1,3,3,1}(q,q,q,q)\\
&\sim 3\frac{J_{2,4}J_{8,16}J_{1,8}J_{6,16}}{\overline{J}_{1,8}\overline{J}_{5,8}} m\Big ( -q^{27},q^{56},*\Big  )
 +3q^{-2}\frac{J_{2,4}J_{8,16}J_{1,8}J_{6,16}}{\overline{J}_{1,8}\overline{J}_{5,8}} m\Big ( -q^{13},q^{56},*\Big  )
  \\
 &\ \ \ \ \  -3q^{-7}\frac{J_{2,4}J_{8,16}J_{3,8}J_{2,16}}{\overline{J}_{1,8}\overline{J}_{5,8}} 
 m\Big ( -q^{-1},q^{56},*\Big  )
 -3q^{-16}\frac{J_{2,4}J_{8,16}J_{3,8}J_{2,16}}{\overline{J}_{1,8}\overline{J}_{5,8}} 
 m\Big ( -q^{-15},q^{56},*\Big  )
 \\
 &\sim 3J_{1,2}\overline{J}_{3,8} m\Big ( -q^{27},q^{56},*\Big  )
 +3q^{-2}J_{1,2}\overline{J}_{3,8} m\Big ( -q^{13},q^{56},*\Big  ) \\
 &\ \ \ \ \  -3q^{-7}J_{1,2}\overline{J}_{1,8} m\Big ( -q^{-1},q^{56},*\Big  )
 -3q^{-16}J_{1,2}\overline{J}_{1,8} m\Big ( -q^{-15},q^{56},*\Big  ),
 \notag
 \end{align*}}%
which leads us to Identity (\ref{equation:g131331-conjecture}).

\section*{Acknowledgements}
This research was supported by Ministry of Science and Higher Education of the Russian Federation, agreement No. 075-15-2019-1619, and by the Theoretical Physics and Mathematics Advancement Foundation BASIS, agreement No. 20-7-1-25-1.

\end{document}